%% file: main.tex
\documentclass{amsart}

\DeclareMathAlphabet{\mathpzc}{OT1}{pzc}{m}{it}

\usepackage{amssymb,amsmath,mathtools,latexsym, amscd, epsfig}
\usepackage{amsthm}
\usepackage{setspace}
\usepackage{graphicx}
\usepackage{mathrsfs}
\usepackage{mdwlist}
\usepackage[dvipsnames,svgnames,x11names]{xcolor}
\usepackage{bbm}
\usepackage{url}
\usepackage[colorlinks=true]{hyperref}
\usepackage[normalem]{ulem}
\usepackage{changes}
\definechangesauthor[color=BrickRed]{EL}

\usepackage{todonotes}
\setcommentmarkup{\todo[color={authorcolor!20},size=\scriptsize]{#3: #1}}

\usepackage[all,cmtip]{xy}
\usepackage[shortlabels]{enumitem}

\input{environments.tex}
\input{macros.tex}

\title{An effective closing lemma for unipotent flows}
\author[E.~Lindenstrauss]{E.~Lindenstrauss}
\address{(E.L.) The Einstein Institute of Mathematics, Edmund J.~Safra Campus, Givat Ram, The Hebrew University of Jerusalem, Jerusalem 91904, Israel}
\email{elon.bl@mail.huji.ac.il}

\author[G.~Margulis]{G.~Margulis}
\address{(G.M.) Mathematics Department, Yale University, New Haven, CT 06511}
\email{gregorii.margulis@yale.edu}

\author[A.~Mohammadi]{A.~Mohammadi}
\address{(A.M.) Department of Mathematics, University of California, San Diego, CA 92093}
\email{ammohammadi@ucsd.edu}
\thanks{A.M.\ acknowledges support by the NSF grants DMS-2055122 and 2350028.}

\author[N.~Shah]{N.A.~Shah}
\address{(N.S.) Department of Mathematics, The Ohio State University, Columbus, OH 43210}
\email{shah@math.osu.edu}

\author[A.~Wieser]{A.~Wieser}
\address{(A.W.) The Einstein Institute of Mathematics, Edmund J.~Safra Campus, Givat Ram, The Hebrew University of Jerusalem, Jerusalem 91904, Israel}
\email{andreas.wieser@mail.huji.ac.il}

\thanks{
E.~L.~and A.~W.~acknowledge the support of ERC 2020 grant HomDyn (grant no.~833423).
}

\renewcommand{\today}{\ifcase \month \or January\or February\or March\or %
April\or May \or June\or July\or August\or September\or October\or November\or %
December\fi, \number \year} 
\date{\today}

\begin{document}

\begin{abstract}
We prove an effective closing lemma for unipotent flows on quotients of perfect real groups.
This is largely motivated by recent developments in effective unipotent dynamics.
\end{abstract}

\maketitle

\setcounter{tocdepth}{1}
\tableofcontents

\section{Introduction}

An important and active research direction in homogeneous dynamics is the quantitative theory of unipotent flows, in particular their equidistribution properties.
The starting point of almost any approach to quantitative equidistribution or density statements is a closing lemma: consider the trajectory of a one-parameter unipotent subgroup $u_t.x$ in a quotient space $G/\Gamma$. Suppose for some (or many) $t_i$ which are well separated we have that the $u_{t_i}.x$ are close to each other. In what sense and under what conditions can we say that $u_{t_i}.x$ are in general position with respect to one another? And if not, what does this tell us about the original point~$x$?

Such issues played a central role already in the pioneering works of G.A. Margulis (G.M.) and of Dani and G.M.\ on the Oppenheim Conjecture (see e.g.~\cite{DM-Oppenheim}) and in more quantitative forms e.g.\ in the works by Einsiedler, Venkatesh, and G.M. in \cite{EMV} and E.L. and G.M. in \cite{LindenstraussMargulis}. More recently, such a closing lemma was used by two of us (E.L.\ and A.M.) in \cite{LM-PolyDensity} and with Zhiren Wang in \cite{LMW22}. Similar statements were also used by Lei Yang in \cite{Lei-SL3}.

The present article seems to us to be of independent interest, but the main impetus to its writing is that it plays an important role in a forthcoming work of A.W.~\cite{AW-realsemisimple}. Loosely, the main result of \cite{AW-realsemisimple} says that a sequence of periodic orbits of a semisimple group (say $H$), inside a quotient $\lquot{\Gamma}{G}$, becomes effectively equidistributed as the volume of the periodic orbit tends to infinity unless there is an obvious obstruction to this. 
The result of \cite{AW-realsemisimple} gives an effective version of (a special case of) a qualitative result by Mozes and N.S.~\cite{Mozes-Shah}. It extends the work of Einsiedler, G.M., and Venkatesh~\cite{EMV}, where a similar result was shown but under the additional requirement that the centralizer of $H$ in $G$ is finite.  
As in~\cite{EMV}, the equidistribution result in \cite{AW-realsemisimple} requires that $G$ be a semisimple algebraic group defined over $\Q$, and $\Gamma$ a \emph{congruence} lattice as it relies on the deep spectral theory available for such quotient spaces $\lquot{\Gamma}{G}$, and more importantly their homogeneous subspaces.

Originally, a closing lemma similar to the one we present here was an ingredient in a yet unpublished work of four of us (the four first-named authors), giving an effective density result for one-parameter unipotent flows (albeit with very poor rates) on arithmetic quotients; another ingredient of that project was an effective avoidance principle that appeared separately (and also seems to us to be of independent interest) by the first four named authors in \cite{LMMS}. We also use the results of \cite{LMMS} in this paper.
For simplicity of the exposition, we limit ourselves to quotients of real groups.

\medskip

Let $\G<\SL_N$ be a perfect $\Q$-group, let $G = \G(\R)$ be its group of real points, and let $\gfrak=\Lie(G)$ be its Lie algebra. Let $\Gamma< G \cap \SL_N(\Z)$ be a lattice in $G$, and let $X = \lquot{\Gamma}{G}$.
Finally, we fix a one-parameter unipotent subgroup $U = \{u_t: t \in \R\}$ of $G$ and consider the action of this group on $X$.

\medskip

At a crucial point in the argument in \cite{AW-realsemisimple} (as well as in the unpublished work of the first four authors mentioned above), one aims to understand whether or not a unipotent orbit can spend a lot of time close to a local orbit of some Lie subgroup of $G$, say with corresponding Lie algebra $\hfrak < \gfrak$.
Note that no further information on the $\hfrak$ is provided; in particular, as opposed to the setting in   \cite{EMV}, the Lie algebra $\hfrak$ might not be semisimple.

The following is a slightly more precise formulation of the question we consider here.
Suppose that $x \in X$ is a point and $T>0$ is a time parameter such that there is a subset $\mathcal{E} \subset [0,T]$ of sufficiently large measure so that for all $s,t \in \mathcal{E}$
\begin{align}\label{eq:question}
xu_s = xu_t g_{st}
\end{align}
where $g_{st} \in G$ is a displacement of bounded (or at least controlled) size which normalizes the given Lie subalgebra $\hfrak<\gfrak$ up to an error of, say, $T^{-1}$.
Of course, if $\hfrak$ is a Lie ideal (or very close to a Lie ideal), this scenario yields no information on the initial point $x$.
Otherwise, we wish to say that $x$ (and in fact the whole unipotent orbit for time $T$) remains close to a periodic orbit of a proper subgroup of $G$ of `low complexity'.

\medskip

We turn to an exact formulation of our main theorems.
As in \cite{LMMS}, we define $\mathcal{H}$ to be the (countable) family of connected $\Q$-subgroups of $\SL_N$ whose radical is unipotent.
We say that $\Mbf < \SL_N$ is of class $\mathcal{H}$ if $\Mbf \in \mathcal{H}$.
For any $\Mbf \in \mathcal{H}$ we put $M = \Mbf(\R)$ and also write $M \in \mathcal{H}$.
We assume throughout the article that $\G \in \mathcal{H}$;
in fact, in the statement of the main theorem, we assume $\G$ is perfect (i.e.~$\G = [\G,\G]$), though this is not used in much of the argument.

We fix once and for all a Euclidean norm $\norm{\cdot}$ on $\Mat_N(\R)$, which induces a norm on $\mathfrak{sl}_N(\R)$ and on $\SL_N(\R)$.
We write  $\norm{\cdot}$ also for the induced norms on exterior products of $\mathfrak{sl}_N(\R)$. For $g \in \SL_N(\R)$ we let
\begin{align*}
\abs{g} = \max\{\norm{g},\norm{g^{-1}}\}.
\end{align*}
For $\tau >0$ let
\begin{align*}
X_{\tau} = \{\Gamma g \in X: \min_{0 \neq \vpz \in \gfrak(\Z)} \norm{\Ad(g)^{-1}\vpz} \geq \tau\}.
\end{align*}
These are compact subsets of $X$, and any compact subset of $X$ is contained in $X_\tau$ for some $\tau>0$.

Let $\gfrak(\Z) = \gfrak \cap \mathfrak{sl}_N(\Z)$, and let $\mathrm{rad}(\gfrak)$ denote the radical of $\gfrak$.
Recall that $U$ is a one-parameter unipotent subgroup of $G$; we write $U = \{u_t = \exp(t\zpz):t \in \R\}$ for a nilpotent unit vector $\zpz \in \gfrak$.

Let $\Mbf \in \mathcal{H}$ be a non-trivial proper subgroup of $\G$ and define 
\begin{align*}
V_M = \wedge^{\dim(M)} \gfrak,\quad \rho_M = \wedge^{\dim(M)} \Ad.
\end{align*}
The representation $\rho_M$ is defined over $\Q$ and the lattice $\bigwedge^{\dim(M)} \gfrak(\Z)$ is $\Gamma$-invariant.
For simplicity, we often write $g.\vpz = \rho_M(g)\vpz$ for the action of $g \in G$ on $\vpz \in V_M$.
Let $\vpz_M \in V_M$ be a primitive integral vector corresponding to the Lie algebra of $M$; that is, $\vpz_M$ is one the two shortest non-zero integral vectors in the line $\wedge^{\dim(M)} \Lie(M)$. So, for any $\gamma\in\Gamma$, $\gamma.\vpz_M=\pm \vpz_{\gamma M \gamma^{-1}}$. 
We write $g\in G \mapsto \eta_M(g) = \rho_M(g^{-1})\vpz_M$ for the (right-)orbit map at $\vpz_M$.
The height of $\Mbf$ is defined to be $\height(\Mbf) = \norm{\vpz_M}$.

For any Lie subalgebra $\hfrak < \gfrak$ (not necessarily defined over $\Q$) we write $\hat{\vpz}_{\hfrak}$ for the point it defines in the projective space $\mathbb{P}(\wedge^{\dim(\hfrak)}\gfrak)$.
For any $0 < r \leq \dim(\gfrak)$ we equip $\mathbb{P}(\wedge^r \gfrak)$ with the Fubini-Study metric $\metric(\cdot,\cdot)$ where the distance
$\metric(\hat{\vpz},\hat{\wpz})$ of two points $\hat{\vpz},\hat{\wpz}\in \mathbb{P}(\wedge^r \gfrak)$ is the angle between the corresponding lines in $\wedge^r \gfrak$. 
If the lines $\hat{\vpz}$ and $\hat{\wpz}$ are represented by pure wedges $v_1\wedge \ldots \wedge v_r$ and $w_1\wedge \ldots \wedge w_r$, respectively, for orthonormal collections of vectors $v_1,\ldots, v_r$ and $w_1,\ldots, w_r$, then 
\begin{align} \label{eq:Fubini-Study}
\begin{split}
\sup_{1\leq i\leq r} \norm{v_i\wedge w_1\wedge\cdots \wedge w_r} & \ll d(\vpz,\wpz)
\ll \sup_{1\leq i\leq r} \norm{v_i\wedge w_1\wedge\cdots \wedge w_r}.
\end{split}
\end{align}

\begin{theorem}\label{thm:main-oneparameter}
Suppose that $\G$ is perfect.
There exist constants $\consta\label{a:mainonep},\consta\label{a:mainonep2}>1$ depending only on $N$, and $E>0$ depending on $N,G,\Gamma$ with the following property.
Let $\hfrak < \gfrak$ be a (proper) subalgebra so that $\hfrak +\operatorname{rad}(\mathfrak g)$ is not a proper Lie ideal of~$\gfrak$. 
Let $\tau\in (0,1)$, $T>0$, and $R>0$ with $T>R>E\tau^{-\ref{a:mainonep}}$.
Let $x = \Gamma g \in X_{\tau}$ be a point.

Suppose that there exists a measurable subset $\mathcal{E} \subset [-T,T]$ with the following properties:
\begin{enumerate}[(a)]
\item $|\mathcal{E}|>TR^{-1/\ref{a:mainonep}}$.
\item For any $s,t\in \mathcal{E}$ there exists $\gamma_{st} \in \Gamma$ with
\begin{gather*}
\norm{u_{-s}g^{-1}\gamma_{st}gu_t}\leq R^{1/\ref{a:mainonep}},\\
\metric\big(u_{-s}g^{-1}\gamma_{st}gu_t. \hat{\vpz}_\hfrak, \hat{\vpz}_\hfrak\big)\leq R^{-1}.
\end{gather*}
\end{enumerate}
Then one of the following is true:
\begin{enumerate}
\item There exists a nontrivial proper subgroup $\Mbf\in\Hcal$ 
so that the following hold for all $t \in [-T,T]$:
\begin{align*}
\norm{\eta_{M}(gu_t)}&\leq R^{\ref{a:mainonep2}},\\
\norm{\zpz\wedge{\eta_{M}(gu_t)}}&\leq T^{-1/\ref{a:mainonep2}}R^{\ref{a:mainonep2}}.
\end{align*}
\item There exist a nontrivial proper normal subgroup $\Mbf\lhd \Gbf$ of class $\Hcal$ containing the radical of $\G$ with
\begin{align*}
\norm{\zpz\wedge \vpz_M}&\leq R^{-1/\ref{a:mainonep2}}.
\end{align*}
\end{enumerate}
\end{theorem}

We also obtain an analogous version for multidimensional unipotent groups.
We follow the setup of \cite[\S2.9]{LMMS}.
Let $U < G$ be a unipotent subgroup and let $\ufrak$ be its Lie algebra.
We fix a basis $\mathcal{B}_U$ of $\ufrak$ consisting of unit vectors 
and set $B_\ufrak(0,\delta) = \{\sum_{\zpz \in \mathcal{B}_U} a_\zpz \zpz: |a_{\zpz}| \leq \delta\}$ for $\delta>0$ as well as $B_U(e) = \exp(B_\ufrak(0,1))$.

Let $\lambda\colon \ufrak \to \ufrak$ be an $\R$-diagonalizable expanding linear map (all eigenvalues have absolute value $>1$).
For any $k\in\Z$ and any $u = \exp(\zpz)\in U$, we set $\lambda_k(u) = \exp(\lambda^k(\zpz))$. We note that $\lambda_k\circ \lambda_\ell=\lambda_{k+\ell}$. 
We shall assume that there exists $k_0 \in \N$ such that for every integer $k>k_0$,
\begin{align}\label{eq:contpropexpmap}
\exp\big(\lambda_{k-k_0}(B_\ufrak(0,1))\big)\exp\big(\lambda_{k-1}(B_\ufrak(0,1))\big)
\subset \exp\big(\lambda_{k}(B_\ufrak(0,1))\big).
\end{align}
Since the exponential map $\exp:\ufrak\to U$ pushes the Lebesgue measure on $\ufrak$ to a Haar measure, denoted by $|\cdot|$, on $U$, for any measurable $B\subset U$ and $k\in\Z$, 
\begin{align} \label{eq:lambda-scale}
    |\lambda_k(B)|=|\det(\lambda)|^k|B|.
\end{align}

The expanding map $\lambda$ could, for instance, be given by an expanding automorphism of $\ufrak$.
Another example is given by expanding the different partial quotients in the lower central series of $\ufrak$ with suitable rates; see \cite[\S2.9]{LMMS}.

To avoid cumbersome statements, we suppose throughout that any constant that is allowed to depend on $N$ and $\height(\G)$ is also (implicitly) allowed to depend on
$\norm{\lambda}$, $\norm{\lambda^{-1}}$, $\frac{|\lambda_1(B_U(e))|}{|B_U(e)|}=|\det(\lambda)|$, and $k_0$.

\medskip

The following is our main theorem in this more general setting.

\begin{theorem}\label{thm:main}
Suppose that $\G$ is perfect.
There exist constants $\consta\label{a:main},\consta\label{a:main2}>1$ depending only on $N$, and $E>0$ depending on $N,G,\Gamma$ with the following property.
Let $\hfrak < \gfrak$ be a (proper) subalgebra so that $\hfrak +\operatorname{rad}(\mathfrak g)$ is not a proper Lie ideal of~$\gfrak$. 
Let $\tau\in (0,1)$, $k \in \N$, and $R>0$ with $\euler^k>R>E\tau^{-\ref{a:main}}$.
Let $x = \Gamma g \in X_{\tau}$ be a point.

Suppose that there exists a measurable subset $\mathcal{E} \subset B_U(e)$ with the following properties:
\begin{enumerate}[(a)]
\item $|\mathcal{E}|>R^{-1/\ref{a:main}}$.
\item For any $u,u'\in \mathcal{E}$ there exists $\gamma \in \Gamma$ with
\begin{gather*}
\norm{\lambda_k(u)^{-1}g^{-1}\gamma g\lambda_k(u')}\leq R^{1/\ref{a:main}},\\
\metric\big(\lambda_k(u)^{-1}g^{-1}\gamma g\lambda_k(u'). \hat{\vpz}_\hfrak, \hat{\vpz}_\hfrak\big)\leq R^{-1}.
\end{gather*}
\end{enumerate}
Then one of the following is true:
\begin{enumerate}[(1)]
\item\label{item:main-option1} There exist a nontrivial proper subgroup $\Mbf\in\Hcal$ 
so that the following hold for all $u \in B_U(e)$:
\begin{equation}\label{eq:main-option1}
\begin{split}
\norm{\eta_{M}(g\lambda_k(u))}&\leq R^{\ref{a:main2}},\\
\max_{\zpz \in \mathcal{B}_U}\norm{\zpz\wedge{\eta_{M}(g\lambda_k(u))}}&\leq \euler^{-k/\ref{a:main2}}R^{\ref{a:main2}}.
\end{split}
\end{equation}
\item\label{item:main-option2} There exist a nontrivial proper normal subgroup $\Mbf\in\Hcal$ containing the radical of $\G$ with
\begin{align*}
\max_{\zpz \in \mathcal{B}_U}\norm{\zpz\wedge \vpz_M}&\leq R^{-1/\ref{a:main2}}.
\end{align*}
\end{enumerate}
\end{theorem}

It is worthwhile noting that the assumptions of the theorem and the conclusion in \eqref{eq:main-option1} only depend on the point $x$ and not the choice of its representative~$g$: The element $h_{u,u'}:=\lambda_k(u)^{-1}g^{-1}\gamma g\lambda_k(u')$ appearing in the assumption (b) is such that $x\lambda_k(u')=x\lambda_k(u)h_{u,u'}$; and the choice of the subgroup $\Mbf$ in Option~(1) of the conclusion depends on $g$, and replacing $g$ by $\gamma_1 g$ for some $\gamma_1\in\Gamma$ exchanges $\Mbf$ for $\gamma_1^{-1}\Mbf \gamma_1$.

This paper is structured as follows.
In \S\ref{sec:Chevalleyandmore}, we establish various facts pertaining to $\Q$-groups generated by lattice elements.
In \S\ref{sec:alm inv Liealgebra}, we show that a Lie algebra $\hfrak$ as in the main theorems cannot be `almost' invariant under a generating set of $G$ consisting of `small' lattice elements.
In \S\ref{sec:proofmain}, we finally prove Theorem~\ref{thm:main} crucially using the results from \cite{LMMS}.

\subsection*{Conventions}
Given $S,T>0$ we write $S \ll T$ when $S \leq c T$ for a constant $c>0$ depending only on $N$ and $S \ll_a T$ when $c>0$ is allowed to depend additionally on an object $a$.
It will also be useful to denote by $\star$ a constant placeholder allowed to depend only on $N$ so that $S \leq T^\star$ if there is $A=A(N)>0$ with $S \leq T^A$.
We use throughout the article constants {\color{red}$A_{\bullet}$} that are only allowed to depend on $N$ and are typically larger than $1$.

\section{Heights of groups generated by lattice elements}\label{sec:Chevalleyandmore}

\subsection{Chevalley representations for groups of class  $\mathcal{H}$}\label{sec:chevalleyclassH}
The following is a classical theorem of Chevalley (see for instance \cite[Ch.~II, Thm.~5.1]{Borelbook} or \cite[Prop.~2.4]{BorelHarishChandra}).

\begin{theorem}
Suppose that $\Mbf < \SL_N$ is a subgroup. Then there exists a rational representation $\rho: \SL_N \to \SL(W)$ and a non-zero vector $\wpz \in W$ such that
\begin{align*}
\Mbf = \{g \in \SL_N: \rho(g)\wpz \wedge \wpz = 0\}.
\end{align*}
If $\Mbf$ is defined over $\Q$, one can find $\rho$ and $\wpz$ also defined over $\Q$.
\end{theorem}

In the following we shall call such a pair $(\rho,\wpz)$ a Chevalley pair for $\Mbf$; for such a pair we will call $\rho$ the Chevalley representation and $\wpz$ the Chevalley vector.
he following proposition allows us to control the height of the Chevalley vector $\wpz$ as well as the occurring representations $\rho$ when $\Mbf \in \mathcal{H}$.

\begin{proposition}[Chevalley representations and heights]\label{prop:chevalley}
There exists a finite collection $\mathcal{P}=\mathcal{P}(N) = \{\rho: \SL_N \to \SL_{\dim(\rho)}\}$ of integral\footnote{Here, we call a rational representation $\rho:\SL_{N}\to \SL_m$ integral if the coefficients of the polynomials defining $\rho$ are integers, and in particular $\rho(\SL_{N}(\Z)) \subset \SL_m(\Z)$.} representations of $\SL_{N}$
and constants $\consta\label{a:chevalley}>0$ and $\constc\label{c:chevalley}>1$ depending only on $N$ with the following property.

Let $\Mbf < \SL_{N}$ be a $\Q$-subgroup of class $\mathcal{H}$.
Then there exists $\rho\in \mathcal{P}$ and a vector $\wpz\in \Z^{\dim(\rho)}$  satisfying
\begin{align*}
\ref{c:chevalley}^{-1}\height(\Mbf)^{1/\ref{a:chevalley}} \leq \norm{\wpz} 
\leq \ref{c:chevalley}\height(\Mbf)^{\ref{a:chevalley}}
\end{align*}
so that $(\rho,\wpz)$ is Chevalley pair for $\Mbf$.
\end{proposition}

It is easy to see that Proposition~\ref{prop:chevalley} cannot be true for all connected $\Q$-groups. Indeed, the proposition implies a uniform bound on the minimal degree of polynomials defining subgroups of class $\mathcal{H}$. 
Such a bound cannot exist for arbitrary $\Q$-subgroups. 
For example, one can take the subtori $\Tbf_d$ of the diagonal torus $\Abf < \SL_3$ given by the additional equation $x_{22}=x_{11}^d$.

To prove Proposition~\ref{prop:chevalley}, we start with the aforementioned uniform degree bound.

\begin{lemma}[Degree bound]\label{lem:bounddegchevalleyH}
There exist an integer $d\geq 1$ depending only on $N$ so that any connected $\Q$-subgroup $\Mbf < \SL_{N}$ of class $\mathcal{H}$  is defined (as a subvariety of $\Mat_{N}$) by rational polynomials of degree at most $d$.

In particular, there exists a finite collection $\mathcal{P}' = \mathcal{P}'(N)$ of integral representations of $\SL_{N}$ such that for any connected subgroup $\Mbf< \SL_{N}$ of class $\mathcal{H}$ there exists $\rho \in \mathcal{P}'$ and $\wpz \in \Z^{\dim(\rho)}$ so that $(\rho,\wpz)$ is a Chevalley pair for $\Mbf$.
\end{lemma}

\begin{proof}
Assuming the degree bound has been established, the second part of the lemma can be deduced as follows. Let $I \subset \Q[\Mat_{N}]$ be the ideal of polynomials vanishing on $\Mbf$ and let $I_d \subset I$ be the subset of polynomials of degree at most $d$.
The regular representation of $\SL_{N}$ on the space of all polynomials $V_d$ on $\SL_{N}$ with degree at most $d$ satisfies that the stabilizer of $I_d$ is $\Mbf$.
The statement of the lemma follows by taking the representation of $\SL_{N}$ on the $\dim(I_d)$-th exterior product of $V_d$. 
We also remark that the representation $\rho$ obtained in this manner satisfies that the entries of $\rho(g)$ are polynomials of degree $\ll_{N} 1$ in the entries of $g$.

It remains to prove the first part of the lemma and, in fact, it suffices to verify that claim over the algebraic closure.
Indeed, let $f$ be a polynomial vanishing on $\Mbf$ and let $\mathcal{A}$ be the finite-dimensional vector space over $\Q$ generated by the coefficients of $f$. Then $f$ can be written as $f = \sum_{i}\alpha_i f_i$ where $\{\alpha_i\}$ is a basis of $\mathcal{A}$ and the polynomials $f_i$ are rational and of degree at most $\deg(f)$. 
By linear independence and Zariski density of $\Mbf(\Q)\subset \Mbf$ (cf.~\cite[Ch.~5, Cor.~18.3]{Borelbook}), the polynomials $f_i$ also vanish on $\Mbf$. 

We have reduced the lemma to the claim that any $\overline{\Q}$-subgroup $\Mbf$ with radical equal to its unipotent radical may be defined by polynomials of degree $\ll_N 1$.
Suppose first that $\Mbf$ is semisimple. 
The representation theory of semisimple Lie algebras (see e.g.~\cite{Jacobson-Liealgebras}) shows that for any semisimple Lie algebra $\hfrak$ the number of $\SL_{N}(\overline{\Q})$-orbits of homomorphisms $\hfrak \to \mathfrak{sl}_{N}$ defined over $\overline{\Q}$ is finite.
In particular, there are finitely many $\SL_{N}(\bar{\Q})$-conjugacy classes of subgroups with Lie algebra isomorphic to $\Lie(\Mbf)$.
Moreover, there are finitely many isomorphism classes of semisimple Lie algebras of dimension at most $\dim(\SL_N)$.
This shows the claim for semisimple subgroups.

If the subgroup $\Mbf$ is unipotent, it is conjugate to a subgroup of the group of upper triangular unipotent matrices $\Ubf<\SL_N$.  
Subgroups of $\Ubf$ correspond to subalgebras of $\Lie(\Ubf)$ via the logarithm map. 
As subalgebras can be defined by linear equations, subgroups of $\Ubf$ can be defined by equations with degree bounded by the degree of the logarithm map (which is $N-1$). 

Suppose now that $\Mbf$ is general. 
As the unipotent case has already been established, let $(\pi,\wpz)$ be a Chevalley pair for the unipotent radical $\Ubf_{\Mbf}$ of $\Mbf$.
In particular, $k = \dim(\pi)\ll_{N}1$.
Let $W \subset \overline{\Q}^k$ be the subspace of vectors fixed by $\Ubf_{\Mbf}$ under $\pi$; by changing basis (over $\overline{\Q}$) we may suppose that $W = \langle e_1,\ldots,e_r\rangle$. 
Thus, $\pi(\Mbf)$ can be viewed as a semisimple subgroup of $\GL_{r}$ after restricting to $W$. By the already proven statement for semisimple subgroups, $\pi(\Mbf)$ can be defined by polynomials equations of degree $\ll_k 1$ and by pullback the same holds for $\Mbf$.
\end{proof}

\begin{proof}[Proof of Proposition~\ref{prop:chevalley}]
Let $\Mbf < \SL_{N}$ be a $\Q$-subgroup of class $\mathcal{H}$ and let $(\rho,\wpz)$ be a Chevalley pair for $\Mbf$ as constructed in Lemma~\ref{lem:bounddegchevalleyH}.
Note that we have no control on the size of $\wpz$ at this point.
Let $W \subset \Q^{\dim(\rho)}$ be the subspace of vectors fixed under $\Mbf(\Q)$.

For a more explicit description, let $\vpz_1,\ldots,\vpz_m$ be a $\Q$-basis of $\mfrak<\mathfrak{sl}_N$ consisting of integral vectors with $\|\vpz_i\| \ll \height(\Mbf)^\star$ (using Minkowski's second theorem). 
Then
\begin{align*}
W = \{\wpz' \in \Q^{k}: \mathrm{D}\rho(\vpz_i)\wpz' = 0 \text{ for all }i \}.
\end{align*}
Then $W$ has height $\height(W)\ll \height(\Mbf)^\star$.
Thus, there is a basis $\wpz_1,\ldots,\wpz_{\dim(W)}$ of $W$ consisting of integral vectors with $\|\wpz_i\|\ll \height(\Mbf)^\star$.
If $g \in \SL_N$ satisfies $\rho(g)\wpz_i=\wpz_i$ for every $i$ then $\rho(g)\wpz = \wpz$ (since $\wpz \in W$) and hence $g \in \Mbf$.

Let $\hat{W}=\bigoplus^{\dim(W)} \Q^{\dim(\rho)}$ where $\bigoplus^{\dim(W)} \Z^{\dim(\rho)}$ is the set of integer points in $\hat{W}$ and the Euclidean norm on $\hat{W}$ is the direct sum of the Euclidean norms in the factors.
Let $\hat{\rho}$ be the representation of $\SL_{N}$ on $\hat{W}$ obtained from $\rho$.
The integral vector $\hat{\wpz} = \wpz_1\oplus \ldots \oplus \wpz_{\dim(W)}$ satisfies $\|\hat{\wpz}\| \ll \height(\Mbf)^\star$ and $\Mbf = \{g \in \SL_N: \hat{\rho}(g)\hat{\wpz} = \hat{\wpz}\}$ by the observations above. This proves the upper bound in the proposition.
For the lower bound, notice that $\mfrak = \{\vpz \in \mathfrak{sl}_N: \mathrm{D}\hat{\rho}(\vpz)\hat{\wpz} = 0\}$ implies $\height(\mfrak) \ll \norm{\hat{\wpz}}^\star$. 
This concludes the proof.
\end{proof}

\subsection{A bound on the number of the connected components}

There is no general bound on the number of connected components of $\Q$-subgroup $\Mbf < \SL_N$.
For an example, one may view for every $k\in \N$ the group of $k$-th roots of unity as $\Q$-subgroups of the multiplicative group $\G_m$.
These finite groups contain very few rational points and, in fact, the following holds.

\begin{lemma}\label{lem:bound on index}
Let $\Mbf \subset \SL_{N}$ be a $\Q$-subgroup with $\Mbf(\Q)\subset \Mbf$ Zariski-dense.
Then 
\begin{align*}
[\Mbf:\Mbf^\circ] \ll_{N} 1
\end{align*}
where $\Mbf^\circ$ is the identity component of $\Mbf$.
\end{lemma}

The proof utilizes the subgroup $\Mbf^{\mathcal{H}}$ defined in \cite{LMMS}: the largest subgroup of $\Mbf$ of class $\mathcal{H}$.
Note that $\Mbf^{\mathcal{H}}$ is a normal subgroup of $\Mbf$.
We also use the following simple lemma.

\begin{lemma}\label{lem:centralizeroftorus}
Let $\Tbf < \SL_{N}$ be a non-trivial torus. 
The centralizer of $\Tbf$ is connected and reductive.
Moreover, the center of the centralizer has at most $O_N(1)$ connected components and the identity component is a torus.
\end{lemma}

\begin{proof}
Without loss of generality, we assume that $\Tbf$ is contained in the diagonal subgroup $\mathbf{A}$.
Let $\chi_{ij}$ for $1 \leq i \neq j \leq N$ be the roots of $\mathbf{A}$ given by $\chi_{ij}(a) = \frac{a_{ii}}{a_{jj}}$. 
Let $\Omega$ be the (possibly empty) set of roots which are trivial on $\Tbf$ and let $\Cbf = \bigcap_{\chi \in \Omega} \ker(\chi) \supset \Tbf$.
After permutation, we may assume that there are $1 = i_0 <i_1< \ldots<i_k=N$ such that 
\begin{align*}
\Cbf = \{a \in \mathbf{A}: a_{ii} = a_{jj} \text{ for all } i_{k'} \leq i,j \leq i_{k'+1}\text{ where } k' \leq k\}.
\end{align*}
A simple calculation shows that the centralizer $\mathbf{L}$ of $\Cbf$ is the subgroup of block matrices consisting of blocks of size $i_1 \times i_1$, $(i_2-i_1) \times (i_2-i_1)$ and so on. In particular, $\mathbf{L}$ is reductive and connected.
The roots outside the diagonal blocks that define $\Lbf$ are non-trivial on $\Tbf$ by construction; hence, the centralizer of $\Tbf$ is $\Lbf$.
Lastly, note that the center $\Cbf$ of $\Lbf$ is isomorphic to 
\begin{align*}
\{(x_1,\ldots,x_k): x_1^{i_1}\cdots x_k^{i_k-i_{k-1}}=1\}.
\end{align*}
Clearly, this variety has at most $O_N(1)$ connected components, and the identity component is a torus.
\end{proof}

\begin{proof}[Proof of Lemma~\ref{lem:bound on index}]
Assume first that $\Mbf^{\circ}$ is trivial.
By classical work of Minkowski \cite{Minkowski} (see also e.g.~\cite{GuralnickLorenz,Serre-orders}), the cardinality of finite subgroups of $\GL_m(\Q)$ is uniformly bounded in terms of $m$.
(The simple bound of $3^{m^2}$ follows from the fact that the kernel of $\GL_m(\Z) \to \GL_m(\Z/3\Z)$ is torsion-free.)
In particular, $|\Mbf(\Q)|\ll_N 1$ and so the claim in this case follows from Zariski-density of $\Mbf(\Q)$ in $\Mbf$.

Suppose now that $\Tbf = \Mbf^\circ$ is a non-trivial $\Q$-torus.
There are finitely many automorphisms $\Tbf \to \Tbf$ obtained through conjugation by elements in $\SL_{N}$, and the number of such automorphisms is bounded in terms of $\dim(\Tbf)$.
Indeed, the normalizer of $\Tbf$ contains the centralizer of $\Tbf$ with finite index \cite[p.~117]{Borelbook} and acts on the group of $\overline{\Q}$-characters of $\Tbf$ with kernel the centralizer of $\Tbf$. 
The group of characters is isomorphic to $\Z^{\dim(\Tbf)}$ and so by the bound on the order of finite subgroups of $\GL_{\dim(\Tbf)}(\Q)$ (discussed at the beginning of the proof) the above claim holds.
Any element of $\Mbf$ defines an automorphism of $\Tbf$ and, after switching to a subgroup of index $\ll_{N} 1$, we may suppose that $\Mbf$ acts trivially on $\Tbf$.
Equivalently, $\Mbf$ is contained in the centralizer $\Lbf$ of $\Tbf$ which is reductive by Lemma~\ref{lem:centralizeroftorus}.

Let $\rho$ be the adjoint representation of $\Lbf$. The image of $\Mbf$ under $\rho$ is a finite subgroup equal to $\rho(\Mbf(\Q))$ by Zariski density.
As in the first case where $\Mbf^\circ$ was trivial, this implies that $|\rho(\Mbf(\Q))| \ll_N 1$.
We may thus replace $\Mbf$ by the kernel of $\rho|_{\Mbf}$ so that $\Mbf$ is contained in the center $\Cbf$ of $\Lbf$.
We may further assume that $\Mbf \subset \Cbf^\circ$ by Lemma~\ref{lem:centralizeroftorus}.

Let $F/\Q$ be a Galois extension splitting the torus $\Cbf^\circ$.
Note that $[F:\Q] \ll_N 1$ since $\Cbf^\circ <\SL_N$.
If $\chi$ is an $F$-character on $\Cbf^\circ$ with $\chi|_{\Tbf} = \mathrm{id}$, then $\chi(\Mbf(\Q)) \subset F^\times$ is a finite subgroup. Necessarily, $\chi(\Mbf(\Q))$ consists of roots of unity in $F$. 
The degree of a primitive root of unity in $F$ is controlled by $[F:\Q]$ and hence $|\chi(\Mbf(\Q))|\ll_{N}1$. 
We apply this discussion to a minimal generating set of $F$-characters for $\Cbf^\circ/\Tbf$ to obtain a morphism $\phi:\Cbf^\circ \to \G_m^r$ defined over $F$ with kernel $\Tbf$ and $|\phi(\Mbf)|\ll_{N}1$ where $r= \dim(\Cbf^\circ)-\dim(\Tbf)$. 
This proves the lemma when $\Mbf^\circ$ is a $\Q$-torus.

Turning to the general case, let $(\rho,\wpz)$ be a Chevalley pair for the group $\Mbf^{\mathcal{H}}$ as found in Lemma~\ref{lem:bounddegchevalleyH}.
Considering the $\Mbf^\mathcal{H}$-fixed vectors under $\rho$, we obtain (by restriction) a representation $\rho': \wholenorm{\Mbf^\mathcal{H}} \to \SL_m$ defined over $\Q$ where $m \leq \dim(\rho)$ and $\ker(\rho') = \Mbf^\mathcal{H}$. 
To conclude the proof, it suffices to show that $[\rho'(\Mbf):\rho'(\Mbf^\circ)]\ll_{m} 1$. But $\rho'(\Mbf^\circ)$ is a (possibly trivial) $\Q$-torus and $\rho'(\Mbf)(\Q) \supset \rho'(\Mbf(\Q))$ is dense in $\rho'(\Mbf)$ so the already proven special cases imply the lemma.
\end{proof}

\subsection{$\Q$-groups generated by lattice elements}\label{sec:heightwhengen}

If  $\gamma \in \SL_N(K)$ for $K$ a number field we define the height of $\gamma$ to be $\height(\gamma)=\prod_v \max\{1,|\gamma_{ij}|_v\}$ where $v$ runs over all places of $K$. For $\gamma \in \SL_N(\Q)$ thisdefinition reduces to $\height(\gamma)=m\max(|\gamma_{ij}|)$, where $m \in \Q_{>0}$ is minimal such that $m \gamma \in \Mat_N(\Z)$.

\begin{proposition}\label{prop:heightifgenbylattice}
There exists $\consta\label{a:heightclassHpart}>0$ depending on $N$ with the following property. 
Let $\gamma_1,\ldots,\gamma_k \in \SL_{N}(\Q)$ with $\height(\gamma_i) \leq T$ for all $i$ and consider the $\Q$-subgroup
\begin{align*}
\Mbf = \overline{\langle \gamma_1,\ldots,\gamma_k\rangle}^z < \SL_N.
\end{align*}
Then $\height(\Mbf^\mathcal{H}) \ll T^{\ref{a:heightclassHpart}}$.
\end{proposition}

We will reduce Proposition ~\ref{prop:heightifgenbylattice} to the case where is $\Mbf$ is connected. 
To that end, we will apply Lemma~\ref{lem:bound on index} which implies the following by a standard argument.

\begin{lemma}\label{lem:generatingsetconncomp}
There exists $\consta\label{a:generatingsetconncomp}>0$ and $k_0\in \N$ depending only on $N$ with the following property.
Let $T>2$ and $\gamma_1,\ldots,\gamma_k \in \SL_{N}(\Q)$ with $\height(\gamma_i)\leq T$.
Set 
\begin{align*}
\Mbf = \overline{\langle \gamma_1,\ldots,\gamma_k\rangle}^z.
\end{align*}
Then there exist $k' \leq k_0k$ and $\eta_1,\ldots,\eta_{k'}\in \SL_{N}(\Q)$ with $\height(\eta_i)\leq T^{\ref{a:generatingsetconncomp}}$ and
\begin{align*}
\Mbf^\circ = \overline{\langle \eta_1,\ldots,\eta_{k'}\rangle}^z.
\end{align*}  
\end{lemma}

\begin{proof}
We begin with an elementary observation:
If $\mathcal{F}$ is a finite group and $\mathcal{S}$ is a generating set of $\mathcal{F}$ containing $1$, then there exists $\ell \leq |\mathcal{F}|$ with $\mathcal{S}^\ell = \mathcal{F}$. Indeed, as $1 \in \mathcal{S}$ we have an ascending sequence $\mathcal{S} \subset \mathcal{S}^2 \subset \mathcal{S}^3\subset \ldots$ and there must be some $\ell \leq |\mathcal{F}|$ with $\mathcal{S}^\ell = \mathcal{S}^{\ell+1}$. But in that case, $\mathcal{S}^\ell$ is invariant under multiplication by $\mathcal{S}$ and, as $\mathcal{S}$ is a generating set, $\mathcal{S}^\ell = \mathcal{F}$.
We apply this discussion to $\mathcal{F} = \Mbf/\Mbf^\circ$.
In view of Lemma~\ref{lem:bound on index} we may replace the generating set $\mathcal{S} = \{\gamma_1,\ldots,\gamma_k\}$ and assume that any coset of $\Mbf^0$ in $\Mbf$ is represented by some element of the generating set.

We fix representives $\id=c_1, c_2,\ldots,c_I \in \mathcal{S}$ of each coset where $I = [\Mbf:\Mbf^0]\ll_N 1$.
For any $1\leq i \leq I$ and $1 \leq j \leq k$ write 
\begin{align}\label{eq:f.i.in f.g.}
c_i\gamma_j = m_{ij}c_{k_{ij}}
\end{align}
for some $k_{ij} \leq I$ and $m_{ij} \in \Mbf^0(\Q)$.
In particular, $\gamma_j = m_{1j}c_{k_{1j}}$.
Note that $\norm{m_{ij}} \ll T^\star$.

We claim that the elements $m_{ij}$ generate a dense subgroup of $\Mbf^0$:
By construction the group $\Delta = \langle \gamma_1,\ldots,\gamma_k\rangle \cap \Mbf^0(\Q)$ is dense in $\Mbf^0$. 
For any word $\omega = \gamma_{j_1}\cdots \gamma_{j_r} \in \Delta$ successive applications of \eqref{eq:f.i.in f.g.} show that $\omega$ is a word in the elements $m_{ij}$ which concludes the claim and hence the lemma.
\end{proof}

\begin{proof}[Proof of Proposition~\ref{prop:heightifgenbylattice}]
By Lemma~\ref{lem:generatingsetconncomp}, we may assume without loss of generality that $\Mbf$ is connected.

\textsc{Case 1: $\mathcal{H}$-groups}.
Suppose that $\Mbf$ is a group of class $\mathcal{H}$.
Let $V_d$ be the space of polynomials on $\Mat_{N}$ of degree at most $d$ and let $I_d \subset V_d$ be the subspace of such polynomials vanishing on $\Mbf$. Here, we take $d\in \N$ to be the degree bound obtained in Lemma~\ref{lem:bounddegchevalleyH}.
It suffices now to show that $\height(I_d) \ll T^\star$.
Set $\gamma_0 = \id$ for convenience and assume without loss of generality that the generating set is invariant under inversion.
For any $m\geq 1$, consider the following linear map:
\begin{align*}
\phi_m: V_d \to \Q^{(k+1)^m},\ p \mapsto (p(\gamma_{i_1}\cdots\gamma_{i_m}))_{i_1,\ldots,i_m=0}^k.
\end{align*}
By construction,
\begin{align*}
\ker(\phi_1) \supset \ker(\phi_2) \supset \ker(\phi_3)\supset \ldots
\end{align*}
Let $m\leq \dim(V_d)$ be minimal such that $\ker(\phi_m) = \ker(\phi_{m+1})$. We claim $\ker(\phi_m)$ is invariant under the action of $\gamma_1,\ldots,\gamma_k$. Indeed let $p\in\ker(\phi_m)=\ker(\phi_{m+1})$, then for any word $w$ of length $m$ in $\{\gamma_i\}$, we have  
\[
\gamma_i.p(w)=p(w\gamma_i)=0.
\]
This implies $\ker(\phi_m)$ is invariant under $\langle \gamma_1, \ldots,\gamma_k\rangle$, hence,  
$\ker(\phi_m) = \ker(\phi_{m+\ell})$ for all $\ell \geq 0$. In particular, 
\[
\ker(\phi_m)=\bigcap_{\ell\geq 1}\ker(\phi_{\ell}).
\]
The above and Zariski-density of $\langle \gamma_1, \ldots,\gamma_k\rangle$ implies that $\ker(\phi_m)=I_d$.  

Note however that since $m\leq \dim(V_d)$, the subspace $\ker (\phi_m)$ has height $\ll T^\star$ so the proposition follows in this case.

\textsc{Case 2: General groups}.
To simplify the discussion we apply the (multiplicative) Jordan decomposition -- see \cite[\S4]{Borelbook}.
For each $\gamma_i$ write $\gamma_i = \gamma_i^s\gamma_i^u=\gamma_i^u\gamma_i^s$ where $\gamma_i^u$ is unipotent and $\gamma_i^s$ is semisimple.
Moreover, $\gamma_i^u,\gamma_i^s \in \Mbf(\Q)$ and, using the fact that $\gamma_i^u,\gamma_i^s$ may be expressed as polynomials in $\gamma_i$ depending on the coefficients of the characteristic polynomial, we have $\height(\gamma_i^u) \ll T^\star$ and $\height(\gamma_i^s) \ll T^\star$.

Consider the subgoup of $\Mbf$ defined as the Zariski closure of the group generated by all commutators $[\gamma_i,\gamma_j]$ and all unipotent elements $\gamma_i^u$. We claim it is equal to $\Mbf^{\mathcal H}$.
By what we have already proved this will show $\height(\Mbf^{\mathcal H})\ll T^{\star}$.

To prove this remaining claim, let $\mathcal{S}_1 = \{[\gamma_i,\gamma_j],\gamma_i^u\}$ and define inductively
\begin{align*}
\mathcal{S}_{\ell+1} = \mathcal{S}_\ell \cup \{[\gamma_i,s]: i \leq k,s \in \mathcal{S}_\ell\}.
\end{align*}
Moreover, we let $\Mbf_\ell$ be the Zariski-closure of the group generated by $\mathcal{S}_\ell$. 
By construction, $\Mbf_1 \subset \Mbf_2\subset \ldots$ and by Lemma~\ref{lem:bound on index} there exists $\ell$ bounded by some constant depending only on $N$ with $\Mbf_{\ell+1}= \Mbf_\ell$.
Then $\Mbf_\ell$ is normal in $\Mbf$ as we have 
\begin{align*}
\gamma_j s \gamma_j^{-1} = [\gamma_j,s]s^{-1} \in \Mbf_{\ell+1} = \Mbf_\ell.
\end{align*}
for any $s \in \mathcal{S}_\ell$ and any $j \leq k$.
By construction, $\Mbf_\ell$ is contained in $\Mbf^\mathcal{H}$ since any character on $\Mbf$ is trivial when restricted to $\Mbf_\ell$.
Moreover, note that the connected group $\Mbf/\Mbf_\ell$ is a torus. 
Indeed, it is abelian as $\gamma_i \gamma_j \Mbf_\ell = \gamma_j \gamma_i \Mbf_\ell$ for all $i,j$ and it is the Zariski closure of the group generated by the semisimple elements $\gamma_i\Mbf_\ell = \gamma_i^s\Mbf_\ell$.
This shows that $\Mbf_\ell = \Mbf^\mathcal{H}$. The proposition now follows from Case 1 applied to $\Mbf_\ell$.
\end{proof}

\section{Almost invariant Lie algebras}\label{sec:alm inv Liealgebra}

The following proposition shows that a Lie algebra which is `almost invariant' under the action of $G$ via the adjoint representation is close to an ideal.
Later in the section, the separation of ideals in the semisimple case will be used to further restrict the options for such a Lie algebra.

\begin{proposition}\label{prop:nearbyideal}
There exists $\consta\label{a:nearbyideal}>0$ and $\constc\label{c:nearbyideal}>1$ depending only on $N$ with the following property.
Let $\G < \SL_N$ be a $\Q$-subgroup of class $\mathcal{H}$.
Let $\hfrak < \gfrak$ be a subalgebra, let $R>0$, and $\delta\in (0,1)$. 
Suppose that $\gamma_1,\ldots,\gamma_k \in \Gamma$ satisfy the following:
\begin{enumerate}[(a)]
\item $\norm{\gamma_i} \leq R$ for all $i=1,\ldots,k$.
\item
For all $i=1,\ldots,k$ we have $\metric(\gamma_i.\hat{\vpz}_\hfrak,\hat{\vpz}_\hfrak) \leq \delta$.
\item  The group generated by $\gamma_1,\ldots,\gamma_k$ is Zariski-dense in $\G$.
\end{enumerate}
Then one of the following holds:
\begin{enumerate}
    \item [(1)] $\delta \geq \ref{c:nearbyideal} \height(\G)^{-\ref{a:nearbyideal}} R^{-\ref{a:nearbyideal}}$.
    \item [(2)] There exists a Lie ideal $\hfrak'\lhd \gfrak$ with $\dim(\hfrak)=\dim(\hfrak')$ and
\begin{align*}
\metric(\hat{\vpz}_{\hfrak},\hat{\vpz}_{\hfrak'})
\leq \delta^{1/\ref{a:nearbyideal}}.
\end{align*}
\end{enumerate}
\end{proposition}

The proof uses a variant of {\L}ojasiewiecz's inequality~\cite{Lojasiewicz}.
{\L}ojasiewiecz's inequality generally asserts that a point with a small value for a real analytic function must be close to its zero locus (see also \cite[Thm.~4.1]{Malgrange}).
Here, we shall use an effective version of this statement for polynomials (in view of the effective dependency on the height of $\G$ in Proposition~\ref{prop:nearbyideal}).
The height of a polynomial $f \in \Z[x_1,\ldots,x_n]$ is the maximum of its coefficients in absolute value.

\begin{theorem}[Solern\'o \cite{Solerno}]\label{thm:effectiveLoj}
For any $d \in \N$ there exists $A(d)>1$ with the following property.

Let $h>1$ and let $f_1,\ldots,f_r \in \Z[x_1,\ldots,x_n]$ have degree at most $d$ and height at most $h$.
Let $V\subset \R^n$ be the zero locus of $f_1,\ldots,f_r$.
Then for $w \in \R^n$
\begin{align*}
\min\{1,\metric(w,V)\} \ll_d (1+\norm{w}_\infty)^{A(d)} h^{A(d)} \max_{1 \leq i\leq r}|f_i(w)|^{1/A(d)},
\end{align*}
where $\metric(w,V)=\inf_{v\in V}\metric(w,v)$, so if $V=\emptyset$, then $\metric(w,V)=\infty$.
\end{theorem}

\begin{proof}
When $V$ is non-empty, the statement follows directly from \cite[Thm.~7]{Solerno}.
So suppose that $V$ is empty. Let $y$ be an auxiliary variable and consider the variety $\Tilde{V} \subset \R^{n+1}$ cut out by the equations $yf_i(x_1,\ldots,x_n) = 0$. Since $V$ is empty, $\Tilde{V}$ is equal to the subspace $\{y=0\}$.
Applying \cite[Thm.~7]{Solerno} for the variety $\Tilde{V}$ and the point $(w,1) \in \R^{n+1}$ for $w \in \R^n$ we obtain 
\begin{align*}
1 = d((w,1),\Tilde{V}) \ll (1+\norm{w}_\infty)^{\star} h^{\star} \max_{1 \leq i\leq r}|f_i(w)|^{\star}
\end{align*}
and the theorem follows.
\end{proof}

The above theorem extends the work of Brownawell \cite{Brownawell88}, who proved the analogous result over $\C$.
We remark that it may also be established using Greenberg's approach \cite{Greenberg1,Greenberg2} as communicated privately to authors by W.~Kim and P.~Yang;  see also \cite[Appendix A]{LMMS} for a related discussion.

\begin{proof}[Proof of Proposition~\ref{prop:nearbyideal}]
Fix an orthonormal basis $\wpz_1,\ldots,\wpz_h$ of $\hfrak$ with respect to a Euclidean norm on $\mathfrak{sl}_N(\R)$. 
By the assumption (a) of the proposition, we have 
\begin{align*}
\max(\norm{\Ad(\gamma_i)},\norm{\Ad(\gamma_i)^{-1}})\ll|\gamma_i|^\ast \ll R^\star.
\end{align*}
Using the assumption (b) and the definition of the Fubini-Study metric (see \eqref{eq:Fubini-Study}), we have for every $i \leq k$ and $\ell \leq h$
\begin{equation}\label{eq:f(w)}
\begin{split}
\norm{\Ad(\gamma_i)\wpz_\ell \wedge \wpz_1 \wedge\ldots \wedge \wpz_h}
&\ll R^\star \Big\| \frac{\Ad(\gamma_i)\wpz_\ell}{\norm{\Ad(\gamma_i)\wpz_\ell}}\wedge \wpz_1 \wedge\ldots \wedge \wpz_h \Big\| \\
&\ll R^\star \delta.
\end{split}
\end{equation}
We apply Theorem~\ref{thm:effectiveLoj} to the variety $V$ of tuples $(\wpz_1',\ldots,\wpz_h')\in \gfrak^h$ with $\Ad(\gamma_i)\wpz_\ell' \wedge \wpz_1' \wedge\ldots \wedge \wpz_h' = 0$ for all $i,\ell$ and $[\wpz_{\ell_1}',\wpz_{\ell_2}'] \wedge \wpz_1' \wedge\ldots \wedge \wpz_h' = 0$ for all $\ell_1,\ell_2$.
Therefore, from \eqref{eq:f(w)}, and since $[\wpz_{\ell_1},\wpz_{\ell_2}] \wedge \wpz_1 \wedge\ldots \wedge \wpz_h = 0$ for all $\ell_1,\ell_2$, we get 
\begin{align*}
\min\{1,\metric((\wpz_1,\ldots,\wpz_h),V)\}
\ll \height(\G)^\star R^\star \delta^\star.
\end{align*}

Suppose that Option~(1) of the conclusion of the proposition does not hold.  Then $\delta \leq \height(\G)^{-A} R^{-A}/C$ for some large $A,C>0$ and, in particular,
\[
\metric((\wpz_1,\ldots,\wpz_h),V)\leq \delta^\star.
\]
Thus, there exists $(\wpz_1',\ldots,\wpz_h') \in V$ with $\norm{\wpz_\ell-\wpz_\ell'} \leq \delta^\star$ for every $\ell$.
By construction, the subalgebra $\hfrak'$ spanned by $\wpz_1',\ldots,\wpz_h'$ is $\Ad(\gamma_i)$-invariant for every $i$ and hence, by the Zariski-density assumption~(c) of the proposition, $\hfrak'$ is $\Ad(\G)$-invariant. In other words, $\hfrak'$ is a Lie ideal, and Option~(2) of the conclusion of the proposition follows.  
\end{proof}

The following lemma asserts that in the semisimple case, a Lie algebra cannot be too close to an ideal unless it is an ideal itself. 

\begin{lemma}\label{lem:insemisimpleidealsrepel}
There exist $\consta\label{a:semisimplenearbyideal}>0$ and $\constc\label{c:semisimplenearbyideal}>0$ depending only on $N$ with the following property.
Suppose that $\gfrak$ is semisimple and that $\hfrak$ is a Lie subalgebra of $\gfrak$ for which there exists a Lie ideal $\hfrak'\lhd \gfrak$ with $\dim(\hfrak)=\dim(\hfrak')$ and 
\begin{align*}
    \metric(\hat{\vpz}_{\hfrak}, \hat{\vpz}_{\hfrak'}) \leq  \ref{c:semisimplenearbyideal}\height(\G)^{-\ref{a:semisimplenearbyideal}}.
\end{align*}
Then $\hfrak=\hfrak'$ and, in particular, $\hfrak$ is a Lie ideal.
\end{lemma}

\begin{proof}
We prove the analogous stronger statement for complex subalgebras of $\gfrak \otimes \C$.

We begin by bringing $\gfrak$ into a `simpler' form (similarly to the proof of Lemma~\ref{lem:bounddegchevalleyH}).
By the representation theory of semisimple Lie algebras (see e.g.~\cite{Jacobson-Liealgebras}), there are finitely many $\SL_N(\overline{\Q})$-conjugacy classes of embeddings of any semisimple $\overline{\Q}$-Lie algebra into $\mathfrak{sl}_N$.
Also, there are finitely many $\overline{\Q}$-isomorphism classes of semisimple Lie algebras that embed into $\mathfrak{sl}_N$.
In particular, there exists a finite list
$\gfrak_1,\ldots,\gfrak_k$ of subalgebras of $\mathfrak{sl}_N$ (which can be taken defined over $\Q$) such that any Lie algebra $\gfrak$ as in the lemma is conjugate to some $\gfrak_i$ over $\overline{\Q}$.
The variety
\begin{align*}
\{g \in \SL_N: \Ad(g)\gfrak = \gfrak_i\}
\end{align*}
is hence non-empty. In view of \cite[Prop.~65]{SG-superapproxI} there must exist $g \in \SL_N(\overline{\Q})$ of height $\ll \height(\G)^\star$ such that $\Ad(g)\gfrak = \gfrak_i$.
In particular, $|g|\ll \height(\G)^\star$.
Considering the Lie algebra $\Ad(g)\hfrak$ the above discussion hence reduces the lemma to the case $\gfrak = \gfrak_i$ and we may ignore the dependency on the height of $\G$ from now on.

Now suppose that $\metric(\hat{\vpz}_{\hfrak}, \hat{\vpz}_{\hfrak'})\leq \delta$ for some $\delta>0$.
Write $\gfrak = \hfrak' \oplus \hfrak''$ for another ideal $\hfrak''$ and $\pi'$, $\pi''$ for the respective projections onto $\hfrak',\hfrak''$.
We have for any $\wpz \in \hfrak$
\begin{align*}
\norm{\pi''(\wpz)} = \norm{\pi'(\wpz)-\wpz}\ll \delta^\star \norm{\wpz}
\end{align*}
by assumption.
In particular, $\pi'|_{\hfrak}$ is an isomorphism for $\delta$ sufficiently small.
The Lie algebra homomorphism $\phi = \pi'' \circ (\pi'|_{\hfrak})^{-1}: \hfrak' \to \hfrak''$ then satisfies $\norm{\phi(\wpz)}\ll \delta^\star \norm{\wpz}$.
We wish to show that $\phi = 0$ unless $\delta $ is not small.

So suppose there is a simple factor $\sfrak$ of $\hfrak' \otimes \C$ with $\phi|_{\sfrak} \neq 0$.
In particular, $\phi|_{\sfrak}$ is injective.
Let $\wpz \in \sfrak$ be any unit vector, $\lambda$ an eigenvalue for the $\ad(\wpz)$-action on $\sfrak$, and $\upz \in \sfrak$ non-zero with $[\wpz,\upz] = \lambda \upz$.
Then
\begin{align*}
|\lambda|\norm{\phi(\upz)} = \big\|[\phi(\wpz),\phi(\upz)]\big\| \leq \norm{\phi(\wpz)}\norm{\phi(\upz)}
\end{align*}
and so $|\lambda| \ll \delta^\star$ since $\phi(\upz) \neq 0$.
Since there are only finitely many ideals $\sfrak$, this is a contradiction if $\delta$ is sufficiently small.
\end{proof}

The following upgrades Proposition~\ref{prop:nearbyideal} to the situation considered in Theorem~\ref{thm:main}.

\begin{proposition}\label{prop:nearbyideal2}
There exists $\consta\label{a:nearbyideal2}>0$ and $\constc\label{c:nearbyideal2}>1$ depending only on $N$ with the following property.
Let $\G < \SL_N$ be a perfect $\Q$-subgroup.
Let $\hfrak < \gfrak$ be a proper subalgebra such that $\hfrak + \mathrm{rad}(\gfrak)$ is not a proper Lie ideal of $\gfrak$.
Let $R\geq 1$, and $\delta\in (0,1)$. 
Suppose that $\gamma_1,\ldots,\gamma_k \in \Gamma$ satisfy the following:
\begin{enumerate}[(a)]
\item $\norm{\gamma_i} \leq R$ for all $i=1,\ldots,k$.
\item\label{item:alminvLie2} For all $i=1,\ldots,k$ we have $\metric(\gamma_i.\hat{\vpz}_\hfrak,\hat{\vpz}_\hfrak) \leq \delta$.
\item The group generated by $\gamma_1,\ldots,\gamma_k$ is Zariski-dense in $\G$.
\end{enumerate}
Then $\delta \geq \ref{c:nearbyideal2} \height(\G)^{-\ref{a:nearbyideal2}} R^{-\ref{a:nearbyideal2}}$.
\end{proposition}

We use the following observation about perfect groups to prove the proposition.

\begin{lemma}\label{lem:idealscontainingLevi}
Suppose that $\gfrak$ is perfect and that $\hfrak \lhd \gfrak$ is a Lie ideal which contains a Levi subalgebra of $\gfrak$. Then $\hfrak= \gfrak$.
\end{lemma}

\begin{proof}
Let $\ufrak_0 = \mathrm{rad}(\gfrak) \rhd \ufrak_1 = [\ufrak_0,\ufrak_0] \rhd \ufrak_2 = [\ufrak_0,\ufrak_1]\rhd \ldots$ be the
lower central series of the (nilpotent) radical of $\gfrak$. 
As $\hfrak$ contains a Levi subalgebra, $\hfrak + \ufrak_0 = \gfrak$. 
Moreover, if $\hfrak +\ufrak_i = \gfrak$ then as $\hfrak$ is an ideal
\begin{align*}
\gfrak = [\gfrak,\gfrak] = [\hfrak,\gfrak] + [\ufrak_i,\hfrak]+[\ufrak_i,\ufrak_i] \subset \hfrak + \ufrak_{i+1}
\end{align*}
and so the claim follows by induction.
\end{proof}

\begin{proof}[Proof of Proposition~\ref{prop:nearbyideal2}]
By Proposition~\ref{prop:nearbyideal} we may either reach the desired conclusion, or there exists a Lie ideal $\hfrak' \lhd \gfrak$ with $\dim(\hfrak') = \dim(\hfrak)$ and $\metric(\hat{\vpz}_{\hfrak},\hat{\vpz}_{\hfrak'})
\leq \delta^{1/\ref{a:nearbyideal}}$.
In particular, the projections of $\hfrak,\hfrak'$ to $\gfrak/\mathrm{rad}(\gfrak)$ are at distance $\ll \height(\G)^\star\delta^{\star}$.
In view of Lemma~\ref{lem:insemisimpleidealsrepel} either $\delta \gg \height(\G)^{-\star}$ (in which case we again conclude) or the projections of $\hfrak,\hfrak'$ to $\gfrak/\mathrm{rad}(\gfrak)$ agree.
In the latter case, by assumption on $\hfrak$ the projection of $\hfrak$ is not a proper ideal and so $\hfrak$ and $\hfrak'$ surject onto $\gfrak/\mathrm{rad}(\gfrak)$. In particular, $\hfrak'$ contains a Levi subalgebra. 
By Lemma~\ref{lem:idealscontainingLevi} we have $\hfrak' = \gfrak$ and so $\hfrak = \gfrak$ which is a contradiction.
\end{proof}

\section{Proof of Theorem~\ref{thm:main}}\label{sec:proofmain}

In the following, we let $\tau \in (0,1)$ and let $ x=\Gamma g \in X_\tau$.
Furthermore, we fix $k \in \N$ and $R\geq 1$, where we assume that $R< \euler^{k}$ and $R>C \tau^{-A}$ for some large constants $C,A$ to be determined in the course of the proof, where $A$ is allowed to depend on $N$ and $C$ is allowed to depend on $N,G,\Gamma$.
Moreover, a proper Lie subalgebra $\hfrak<\gfrak$ is given with the property that $\hfrak+\mathrm{rad}(\gfrak)$ is not a proper Lie ideal of $\gfrak$.

We suppose that there exists a measurable subset $\mathcal{E} \subset B_U(e)$ with the following:
\begin{enumerate}[label={\color{red}(A\arabic*)}]
\item\label{item:assumptionE1} $|\mathcal{E}|>R^{-\kappa}$.
\item\label{item:assumptionE2} For any $u,w\in \mathcal{E}$ there exists $\gamma \in \Gamma$ with 
\begin{gather*}
\norm{\lambda_k(w)^{-1}g^{-1}\gamma g\lambda_k(u)}\leq R^{\kappa},\\
\metric(\lambda_k(w)^{-1}g^{-1}\gamma g\lambda_k(u). \hat{\vpz}_\hfrak, \hat{\vpz}_\hfrak)\leq R^{-1}.
\end{gather*}
\end{enumerate}
Here, a small $\kappa>0$, depending only on $N$, will be determined during the proof; and we will choose $A$ large enough so that $\kappa A\geq 2k_0$, and hence $R^\kappa>\euler^{2k_0}$, where $k_0$ is as in \eqref{eq:contpropexpmap}.

Throughout the proof, we will use the notation from the introduction.

The main step in the proof of Theorem~\ref{thm:main} (namely Proposition~\ref{prop:mainstep} below) consists of showing that there is a `large' set of points in $B_U(e)$ for which the associated point along the orbit through $x$ is not `Diophantine' (cf.~\cite[Def.~3.1]{LMMS}) in a specific sense.
We then apply the main theorem of \cite{LMMS}.

\subsection{Visits to cusp neighborhoods}
We first shrink $\mathcal{E}$ slightly to control the height in the cusp.
To that end, one wishes to apply non-divergence results for unipotent flows --- see for instance the papers \cite{Dani-Nondiv-2,DM-Linearization,KM-Nondiv}.
We will use the following effective version of the non-divergence theorem proven in \cite{LMMS}.

\begin{theorem}[{\cite[Theorem~6.3]{LMMS}}]\label{thm:nondiv}
There exist a constant $\consta\label{a:nondiv}>0$ depending on $N$ and a constant $\constE\label{c:nondiv}>0$ depending polynomially on $\height(\G)$ with the following property.

For any $g \in G$, $\eta \in (0,1/2)$, and $k \geq 1$ at least one the following holds:
\begin{enumerate}
\item We have
\begin{align*}
|\{u \in B_U(e): \Gamma g\lambda_k(u) \not\in X_\eta\}| \leq \ref{c:nondiv} \eta^{1/\ref{a:nondiv}}.
\end{align*}
\item There exists a non-trivial proper subgroup $\Mbf \in \mathcal{H}$ such that for all ${u \in B_U(e)}$
\begin{align*}
\norm{\eta_M(g \lambda_k(u)} &\leq \ref{c:nondiv} |g|^{\ref{a:nondiv}} \eta^{1/\ref{a:nondiv}},\\
\max_{\zpz \in \mathcal{B}_U}\norm{\zpz \wedge \eta_M(g \lambda_k(u)}  &\leq \ref{c:nondiv} |g|^{\ref{a:nondiv}} \eta^{1/\ref{a:nondiv}} \euler^{-k/\ref{a:nondiv}}.
\end{align*}
Moreover, $\Mbf = \Nbf_\Ubf^\mathcal{H}$ for some unipotent subgroup $\Ubf < \G$ where $\Nbf_\Ubf$ is the normalizer of $\Ubf$.
\end{enumerate}
\end{theorem}

By \cite[Lemma 2.8]{LMMS}, there exists $F>0$ (depending only on $N$) and $E_{\G}>0$ depending on the geometry of $\Gamma \backslash G$ so that for any $y \in X_\tau$ we can find $g_0 \in G$ with $\Gamma g_0 = y$ and 
\begin{align}\label{eq:diamestimate}
|g_0| \leq E_\G \tau^{-F};
\end{align}
see also \cite{MSGT-diameter} for more precise statements.
 
Applying Theorem~\ref{thm:nondiv} with $\eta = R^{-4\ref{a:nondiv}\kappa}$, our given point $x \in X_\tau$, and a representative $g \in G$ of $x$ as in \eqref{eq:diamestimate} we have 
\begin{align*}
|\{u \in B_U(e): x\lambda_k(u) \not\in X_{R^{-4\ref{a:nondiv}\kappa}}\}| \leq R^{-2\kappa}
\end{align*}
or Option~(2) of Theorem~\ref{thm:nondiv} holds.
In the latter case Option~(1) of the conclusion of Theorem~\ref{thm:main} holds, because due to our assumption that $R \gg \tau^{-\star}$ we have $|g|\leq E_\G \tau^{-F}\ll R^{\star}$.
Recall here that, in view of the comment after Theorem~\ref{thm:main}, it is enough to verify the conclusion of Theorem~\ref{thm:main} for any choice of a representative $g$ of $x$.

In particular, we assume from now on --- after  discarding a subset of $\mathcal{E}$ of size at most $R^{-2\kappa}$  --- that the following holds:

\begin{enumerate}[label={\color{red}(A\arabic*)}]
\setcounter{enumi}{2}
    \item\label{item:assumptionE3}
    For any $u \in \mathcal{E}$ we have $x\lambda_k(u)\in X_{R^{-4\ref{a:nondiv}\kappa}}$.
\end{enumerate}

\subsection{A large set of non-Diophantine points}\label{sec:manynon-Dio}
Assume that $\mathcal{E}$ is such that \ref{item:assumptionE1}, \ref{item:assumptionE2}, and \ref{item:assumptionE3} hold for some $\kappa>0$ (fixed, but small).
We further fix some $\theta\in (0,1)$, to be determined in the course of the proof, with $\theta^{\dim(\G)+1} > \kappa$. 
Set
\begin{align}\label{eq:define elli}
\ell_i = \lfloor\theta^{\dim(\G)+1-i}\log(R) \rfloor,\quad \text{for $i = 0,\ldots,\dim(\G)+1$.}
\end{align}
Note that $\ell_{i+1}\geq \lfloor \theta^{-1}\rfloor \ell_i$, and  $2k_0\leq \ell_0 \leq \ldots\leq\ell_{\dim(\G)+1} <k$.
The aim of this subsection is to prove the following claim that plays a key role in the proof of Theorem~\ref{thm:main}.

\begin{proposition}\label{prop:mainstep}
For any $D>1$ there exist $\kappa,\theta \in (0,1)$ and $\consta\label{a:mainstep}>1$ (depending only on $N,D$), and a measurable subset $\mathcal{E}_0 \subset \mathcal{E}$ with $|\mathcal{E}_0| \gg |\mathcal{E}|$ and with the following property: 
for any $w \in \mathcal{E}_0$ there exists $i\in \{0,\ldots,\dim(\G)\}$ and a proper $\Q$-subgroup $\Lbf < \G$ of class $\mathcal{H}$ such that
\begin{align}\label{eq:mainstep}
\begin{split}
\norm{\eta_L(g\lambda_k(w))} &\leq \euler^{ \ell_{i}\ref{a:mainstep}},\\
\max_{\zpz \in \mathcal{B}_U} \norm{\zpz \wedge \eta_L(g\lambda_k(w))} &\leq \euler^{-D\ell_{i}\ref{a:mainstep} }.
\end{split}
\end{align}
\end{proposition}

\noindent
The rest of the subsection is dedicated to the proof of Proposition~\ref{prop:mainstep}.
As before, $\theta,\kappa$ are fixed small constants to be determined during the proof.

\subsubsection{A set of good density points}
The following technical lemma constructs the desired subset $\mathcal{E}_0$ of points with good density simultaneously for all scales $\ell_i$.

\begin{lemma}\label{lem:goodsubset}
There exists a measurable subset $\mathcal{E}_0 \subset \mathcal{E}$ with $|\mathcal{E}_0| \gg |\mathcal{E}|$ such that for every $i=0,\ldots,\dim(\G)+1$ and every $w \in \mathcal{E}_0$ we have
\begin{align}\label{eq:gooddensitypts}
|\{u \in \mathcal{E}:
\lambda_k(u)\in \lambda_k(w) \lambda_{\ell_{i}}(B_U(e))\}|
 \gg|\mathcal{E}| \frac{|\lambda_{\ell_{i}}(B_U(e))|}{|\lambda_{k}(B_U(e))|}
\end{align}
as well as
\begin{align}\label{eq:gooddensitypts2}
|\{u \in B_U(e): \lambda_k(w)\lambda_{\ell_i}(u) \in \lambda_k(\mathcal{E})\}|
\gg |\mathcal{E}|.
\end{align}
\end{lemma}

\begin{proof}
In the following, we will use the fact that for all $\ell\geq 1$, $\lambda_\ell(\cdot)$ scales the Haar measure on $U$ by a constant factor (indeed, the Haar measure on $U$ corresponds to the Lebesgue measure on the Lie algebra $\ufrak$ via the exponential map).

We construct by downward induction on $i$ subsets $\mathcal{E}_i$ so that $\mathcal{E}_i$ satisfies the lemma for all $i' \geq i$.
For $i = \dim(\G)+1$, let $\mathcal{U}'\subset B_U(e)$ be a maximal subset such that the translates $\lambda_k(u) \lambda_{\ell_{\dim(\G)+1}-2k_0}(B_U(e))$ for $u \in \mathcal{U}'$ are pairwise disjoint, where $k_0>0$ is as in \eqref{eq:contpropexpmap}.
Then
\begin{align} \label{eq:cover}
\lambda_k(B_U(e))\subset \bigcup_{u \in \mathcal{U}'} \lambda_k(u) \lambda_{\ell_{\dim(\G)+1}-k_0}(B_U(e)),
\end{align}
where the multiplicity of the covering is $\ll1$. 
In particular, we have
\begin{align*}
\#\mathcal{U}' \asymp \frac{|\lambda_k(B_U(e))|}{|\lambda_{\ell_{\dim(\G)+1}}(B_U(e))|}.
\end{align*}
Let $\mathcal{U} \subset \mathcal{U}'$ be the subset of points $u$ for which 
\begin{equation}\label{eq:E0defU'}
   \begin{aligned}
|\lambda_k(\mathcal{E}) \cap \lambda_k(u) \lambda_{\ell_{\dim(\G)+1}-k_0}(B_U(e))| &\geq \frac{|\lambda_k(\mathcal{E})|}{2\cdot (\#\mathcal{U}')}\\
&\gg |\mathcal E| |\lambda_{\ell_{\dim(\G)+1}}(B_U(e))|,
\end{aligned} 
\end{equation}
where we used $|\mathcal E|=|\lambda_k(\mathcal E)|/|\lambda_k(B_U(e))|$
in the second inequality. Set
\begin{align*}
\mathcal{E}_{\dim(\G)+1}:= \{u' \in \mathcal{E}: 
\exists u \in \mathcal{U} \text{ s.t. }
\lambda_k(u')\in \lambda_k(u) \lambda_{\ell_{\dim(\G)+1}-k_0}(B_U(e)) \}.
\end{align*}
Then, by \eqref{eq:cover} and the first inequality in \eqref{eq:E0defU'}, 
$|\lambda_k(\mathcal{E}_{\dim(\G)+1})|\geq \frac{1}{2}|\lambda_k(\mathcal{E})|$, and hence
$|\mathcal{E}_{\dim(\G)+1}| \geq \frac{1}{2}|\mathcal{E}|$.
Moreover, if $w \in \mathcal{E}_{\dim(\G)+1}$ and $u \in \mathcal{U}$ is such that $\lambda_k(w)\in \lambda_k(u) \lambda_{\ell_{\dim(\G)+1}-k_0}(B_U(e))$, then by \eqref{eq:contpropexpmap}
\begin{align}\label{eq:E0constr-restrtoE}
\lambda_k(u) \lambda_{\ell_{\dim(\G)+1}-k_0}(B_U(e)) \subset
\lambda_k(w) \lambda_{\ell_{\dim(\G)+1}}(B_U(e)).
\end{align}

Thus by \eqref{eq:E0constr-restrtoE} and \eqref{eq:E0defU'}, we have 
\begin{equation}\label{eq:gooddensitypts'}
|\lambda_k(\mathcal{E}) \cap \lambda_k(w)\lambda_{\ell_{\dim(\G)+1}}(B_U(e))|
\gg |\mathcal{E}| |\lambda_{\ell_{\dim(\G)+1}}(B_U(e))|
\end{equation}
which implies \eqref{eq:gooddensitypts}; recall that $\lambda_k$ scales the Haar measure by a constant factor.  

Since the Haar measure is left invariant,~\eqref{eq:gooddensitypts'} also implies that for $w \in \mathcal{E}_{\dim(\G)+1}$
\begin{align*}
|\lambda_k(w)^{-1}\lambda_k(\mathcal{E}) \cap \lambda_{\ell_{\dim(\G)+1}}(B_U(e))|
\gg |\mathcal{E}| |\lambda_{\ell_{\dim(\G)+1}}(B_U(e))|
\end{align*}
and so \eqref{eq:gooddensitypts2} holds for $i = \dim(\G)+1$; recall again that $\lambda_{\ell_{\dim(\G)+1}}$ scales the Haar measure by a constant factor.

We iterate the above construction using at the $i$-th step the set $\mathcal{E}_{i+1}$ constructed in the previous step instead of $\mathcal{E}$.
This constructs a set $\mathcal{E}_0$ with the desired properties and so yields the lemma.
\end{proof}

\subsubsection{Constructing a group of class $\mathcal{H}$}\label{sec:construct-group}
\textbf{We fix $w \in \mathcal{E}_0$ till we complete the proof of Proposition~\ref{prop:mainstep}.}

Let $\gamma_0 \in \Gamma$ be such that for
\begin{align} \label{eq:g0}
    g_0 = \gamma_0 g \lambda_k(w) \text{, we have } |g_0| \ll R^{\star \kappa},
\end{align}
which is possible in view of \eqref{eq:diamestimate} and \ref{item:assumptionE3}.
By \ref{item:assumptionE2}, and inserting $g_0 = \gamma_0 g \lambda_k(w)$, for any $u \in \mathcal{E}$ we pick $\gamma_u \in \Gamma$ with
\begin{align}\label{eq:deflatticeelements}
\norm{g_0^{-1} \gamma_u g_0 \lambda_k(w)^{-1}\lambda_k(u)} &\leq R^{\kappa},       \\
\metric(g_0^{-1}\gamma_u g_0\lambda_k(w)^{-1}\lambda_k(u). \hat{\vpz}_\hfrak, \hat{\vpz}_\hfrak)&\leq R^{-1}\label{eq:deflatticeelements2}.
\end{align}
To simplify notation, we set 
\begin{align} \label{eq:gu}
g_u = g_0^{-1} \gamma_u g_0 \lambda_k(w)^{-1}\lambda_k(u).
\end{align}
For $0\leq i\leq \dim(\G)$, let
\begin{align} \label{eq:Eprime_i}
\mathcal{E}_i' = \{u \in \mathcal{E}:  \lambda_k(u) \in \lambda_k(w)\lambda_{\ell_i}(B_U(e))\}.
\end{align}
Then $\mathcal{E}_0' \subset \mathcal{E}_1'\subset \ldots$.
For any $u \in \mathcal{E}_i'$ we have 
\begin{align}\label{eq:heightboundlatticel}
\norm{\gamma_u} \leq \euler^{\ref{a:boundlatticeel}\ell_i}
\end{align}
for some $\consta\label{a:boundlatticeel}>0$ by \eqref{eq:g0}, \eqref{eq:deflatticeelements}, and that $R^\kappa \leq \euler^{\ell_0}\leq \euler^{\ell_i}$. We set
\begin{align*}
\Mbf_i = \overline{\langle\gamma_u: u \in \mathcal{E}_i'\rangle}^z.
\end{align*}
In the following, we first collect some properties of the subgroups $\Mbf_i$.

\begin{lemma}\label{lem:gen group proper}
For every $0\leq i\leq \dim(\G)$, the subgroup $\Mbf_i<\G$ is proper.
\end{lemma}

To prove Lemma~\ref{lem:gen group proper}, we use the following lemma, which essentially asserts that $\hfrak$ is almost $U$-normalized.

\begin{lemma}\label{lem:Liealg almost U normalized}
If $\theta$ is sufficiently small depending only on $N$, for all $u \in B_U(e)$, 
\begin{align*}
\metric(\lambda_{\lfloor \theta \log(R)\rfloor}(u).\hat{\vpz}_\hfrak,\hat{\vpz}_\hfrak) \ll R^{-1/4}.
\end{align*}
\end{lemma}

\begin{proof}[Proof of Lemma~\ref{lem:Liealg almost U normalized}]
Let $\ell = \lfloor \theta \log(R)\rfloor$ (note that $\ell=\ell_{\dim(\G)}$ by \eqref{eq:define elli}) and set
\begin{align*}
\mathcal{F}' = \{u \in B_U(e): \lambda_k(w)\lambda_\ell(u) \in \lambda_k(\mathcal{E})\}.
\end{align*}
Then, for each $u\in\mathcal{F}'$, we have a unique $u' \in \mathcal{E}$ such that   
\begin{align}\label{eq:u and u'}
\lambda_k(w)\lambda_\ell(u) = \lambda_k(u')
\end{align}
By \eqref{eq:gooddensitypts2}, we have $|\mathcal{F}'| \gg |\mathcal{E}|$, since $w \in \mathcal{E}_0$.

As the number of lattice elements $\gamma \in \Gamma<\SL(N,\Z)$ with $\norm{\gamma} \leq \euler^{\ref{a:boundlatticeel} \ell}$ is $\ll \euler^{\star \ell}$, 
it follows from \eqref{eq:heightboundlatticel} that there exists a subset $\mathcal{F}\subset\mathcal{F}'$ of measure $|\mathcal{F}| \gg |\mathcal{E}|\euler^{-\star \ell}$ on which the map $u\in \mathcal{F}\mapsto \gamma_{u'}$ is constant, where $u'$ is given by \eqref{eq:u and u'} and $\gamma_{u'}$ is as in \eqref{eq:gu}.

We fix $u_0 \in \mathcal{F}$. By \eqref{eq:deflatticeelements} and \eqref{eq:gu}, $\norm{g_{u_0'}}\leq R^\kappa$. For any $u'\in \mathcal{E}$, by \eqref{eq:deflatticeelements2} and \eqref{eq:gu}, we have $\metric(g_{u'}.\hat{\vpz}_{\hfrak}, \hat{\vpz}_{\hfrak}) \ll R^{-1}$.
Therefore, by the argument as in \eqref{eq:f(w)}, we get
\[
\metric(g_{u_0'}^{-1}g_{u'}.\hat{\vpz}_{\hfrak}, \hat{\vpz}_{\hfrak}) \ll R^{\star \kappa} R^{-1}\ll R^{-1/2},
\]
where we assume that $\kappa$ is sufficiently small depending on $N$. 

On the other hand, for any $u \in \mathcal{F}$ and the corresponding $u' \in \mathcal{E}$ as in \eqref{eq:u and u'},
\begin{align*}
g_{u_0'}^{-1}g_{u'} = \lambda_{k}(u_0')^{-1} \lambda_k(u') = \lambda_\ell(u_0)^{-1}\lambda_\ell(u) \in \lambda_{\ell+k_0}(B_U(e)),
\end{align*}
where $g_{u_0},g_{u}$ are given by \eqref{eq:gu}, $\gamma_{u_0'}=\gamma_{u'}$ by the definition of $\mathcal{F}$, and we use \eqref{eq:contpropexpmap} after noting that by definition $\lambda_\ell(u_0)^{-1}=\lambda_\ell(u_0^{-1})$. 
Let 
\begin{align*}
    \mathcal{G}_{u_0}=\{u_1 \in B_U(e):\lambda_{\ell+k_0}(u_1)=\lambda_{\ell}(u_0)^{-1} \lambda_\ell(u) \text{ for some } u \in \mathcal{F}\}.
\end{align*}
Then, by scale invariance by $\lambda_\ell$ of the Haar measure on $U$, 
\begin{align*}
    |\mathcal{G}_{u_0}|=|\det(\lambda)|^{-k_0} |\mathcal{F}|\gg |\mathcal{E}|\euler^{-\star \ell} \gg \euler^{-\star \ell},
\end{align*}
and for all $u_1\in \mathcal{G}_{u_0}$, 
\begin{align}\label{eq:HalmostUinv}
\metric(\lambda_{\ell+k_0}(u_1).\hat{\vpz}_{\hfrak}, \hat{\vpz}_{\hfrak}) \ll R^{-1/2}.
\end{align}
For any $u_1\in B_u(e)$, $|\lambda_{\ell+k_0}(u_1)|\ll \euler^{\ast \ell}$, and hence by \eqref{eq:Fubini-Study}, the left hand side of \eqref{eq:HalmostUinv} can be recast in terms of polynomials of degree, say at most $d$, depending only on $N$: given an orthonormal basis $\wpz_1,\ldots,\wpz_h$ of $\hfrak$, we have for any $u_1 \in B_U(e)$
\begin{align*}
\euler^{-\star \ell} \max_{i=1,\ldots,h} &\norm{\lambda_{\ell+k_0}(u_1).\wpz_i \wedge \wpz_1\wedge \ldots \wedge \wpz_h}\\
&\ll \metric(\lambda_{\ell+k_0}(u_1).\hat{\vpz}_{\hfrak}, \hat{\vpz}_{\hfrak})
\ll \euler^{\star \ell} \max_{i=1,\ldots,h} \norm{\lambda_{\ell+k_0}(u_1).\wpz_i \wedge \wpz_1\wedge \ldots \wedge \wpz_h}.
\end{align*}
Here, the lower bound follows as in \eqref{eq:f(w)} and the upper bounds follows from Gram-Schmidt applied to the vectors $\lambda_{\ell+k_0}(u_1).\wpz_i$, for $1 \leq i \leq h$.
Thus, the Remez inequality (which we recall in Lemma~\ref{lem:Remez} below) implies that
\begin{align} \label{eq:Gu0toB}
    \sup_{u_1\in B_U(e)} \metric(\lambda_{\ell+k_0}(u_1).\hat{\vpz}_{\hfrak}, \hat{\vpz}_{\hfrak}) \ll e^{\star \ell} R^{-1/2},
\end{align}
where $\star$ is a large constant no more than a fixed power of $N$.  By the definition of $\ell$, $e^\ell\leq R^\theta$. So, by choosing $\theta$ sufficiently small, depending only on $N$, the conclusion of the lemma follows. 
\end{proof}

The following is a special case of \cite[Lemma 5.4]{LMMS}.

\begin{lemma}[Remez inequality]\label{lem:Remez}
Let $f_1,\ldots,f_r\in \R[x_1,\ldots,x_n]$ be nonzero polynomials of degree at most $d$ and set $f(x) = \max_j |f_j(x)|$. 
For any compact, convex subset $B \subset \R^n$ and any $\delta>0$ we have
\begin{align*}
\big|\big\{x \in B: f(x) < \delta \sup_{y \in B} f(y)\}
\big\}\big| \leq c \delta^{\frac{1}{d}}|B|
\end{align*}
where $c>0$ depends only on $r,n,d$.
\end{lemma}

\begin{proof}[Proof of Lemma~\ref{lem:gen group proper}] Let $u\in \mathcal{E}_i'$, then $\lambda_k(w)^{-1}\lambda_k(u)=\lambda_{\ell_i}(u_1)$ for some $u_1\in B_U(e)$, see \eqref{eq:Eprime_i}. By \eqref{eq:deflatticeelements2} 
\begin{align} \label{eq:small-difference}
    d(g_0^{-1} \gamma_u g_0 \lambda_{\ell_i}(u_1).\hat{\vpz}_\hfrak,\hat{\vpz}_\hfrak)\leq R^{-1}.
\end{align}
And by Lemma~\ref{lem:Liealg almost U normalized}, 
\begin{align} \label{eq:almost-fixed}
    d(\lambda_{\ell_i}(u_1).\hat{\vpz}_\hfrak,\hat{\vpz}_\hfrak)\leq R^{-1/4}.
\end{align}
We have that $\norm{g_0}\ll R^{\star \kappa}\ll R^{\ast\theta}$ by \eqref{eq:g0} and $\norm{\gamma_u}\leq e^{\ref{a:boundlatticeel}\ell_i}\leq R^{\ast\theta}$ by \eqref{eq:heightboundlatticel}. Therefore, from \eqref{eq:almost-fixed} we get
\begin{align*}
    d(g_0^{-1} \gamma_u g_0 \lambda_{\ell_i}(u_1).\hat{\vpz}_\hfrak,g_0^{-1} \gamma_u g_0. \hat{\vpz}_\hfrak)\ll R^{\star \theta} R^{-1/4}. 
\end{align*}
Combining this with \eqref{eq:small-difference} we get
\begin{align*}
    d(g_0^{-1} \gamma_u g_0.\hat{\vpz}_\hfrak,\hat{\vpz}_\hfrak)\ll R^{\star \theta} R^{-1/4}+R^{-1}.
\end{align*}
Hence, 
\begin{align} \label{eq:gamma-almost-fixed}
    d(\gamma_u g_0.\hat{\vpz}_\hfrak,g_0.\hat{\vpz}_\hfrak)\ll R^{\ast \kappa}(R^{\star \theta} R^{-1/4}+R^{-1})\leq R^{-1/8}
\end{align}
for sufficiently small $\theta$, depending only on $N$.

We consider Proposition~\ref{prop:nearbyideal2} for the subalgebra  $\Ad(g_0)\hfrak$ and the finite set $\{\gamma_u:u\in\mathcal{E}_i'\}\subset\Gamma$. Then its conditions (a) and (b) are satisfied for $\euler^{\ref{a:boundlatticeel}\ell_i}$ in place of $R$ and $R^{-1/8}$ in place of $\delta$. We choose $\theta$ sufficiently small depending only on $N$, and we choose $C$ sufficiently large depending on $N$, $\G$ and $\Gamma$, so that since $R>C\tau^{-A}$, we get
\[
\ref{c:nearbyideal2} \height(\G)^{-\ref{a:nearbyideal2}} e^{-\ref{a:nearbyideal2}\ref{a:boundlatticeel}\ell_i}\geq  \ref{c:nearbyideal2} \height(\G)^{-\ref{a:nearbyideal2}} R^{-\ref{a:nearbyideal2}\ref{a:boundlatticeel}\theta}>R^{-1/8}=\delta.
\]
Then, the conclusion of Proposition~\ref{prop:nearbyideal2} fails to hold, and hence its condition (c) cannot hold. Therefore, $\Mbf_i$, which is the Zariski closure of $\{\gamma_u:u\in\mathcal{E}_i'\}$, cannot contain $\G$.
\end{proof}

\begin{lemma}\label{lem:classHpart nontrivial}
For all $0\leq i\leq\dim\G$, the group $\Mbf_i^{\mathcal{H}}$ is non-trivial.
\end{lemma}

\begin{proof}
We will show that the number of lattice points in $\Mbf_i^\circ(\R)$ of norm at most $K$ is at least polynomial in $K$ for some $K = \euler^{\star \ell_i}$.
Note that for any $\Q$-torus $\Tbf< \SL_{N}$ and any $K\geq 2$ we have
\begin{align}\label{eq:latticeptcountintori}
\#\{\gamma \in \Tbf(\Z): \norm{\gamma}\leq K\} \ll_{N} \log(K)^{\star};
\end{align}
see for instance \cite[Lemma 6.3]{EL23-nonmaximal}.
Thus, the claimed lattice point estimate implies that $\Mbf_i^\circ$ is not a torus, and so $\Mbf_i^{\mathcal{H}}$ is nontrivial.

Notice first that $\gamma_u = \gamma_{u'}$ for $u,u' \in \mathcal{E}_i'$ implies $\norm{\lambda_k(u)^{-1}\lambda_k(u')} \ll R^{\star \kappa}$ by \eqref{eq:deflatticeelements} and, in particular,
$\lambda_k(u') \in \lambda_k(u) \lambda_{\star \kappa \log(R) }(B_U(e))$.
For any $c>0$,
\begin{align*}
|\{u' \in B_U(e):\lambda_k(u') \in \lambda_k(u) \lambda_{c \kappa \log(R) }(B_U(e))\}| \ll \frac{R^{\star c \kappa}}{|\lambda_k(B_U(e))|}.
\end{align*}
Now since $w\in\mathcal E_0$, it follows from \eqref{eq:gooddensitypts} that
\begin{align}  \label{eq:volEiprime}
|\mathcal E_i'|\gg |\mathcal E|\frac{|\lambda_{\ell_i}(B_U(e))|}{|\lambda_k(B_U(e))|}\gg R^{-\kappa}\frac{\euler^{\star \ell_i}}{|\lambda_k(B_U(e))|}.
\end{align}
Together, the above estimates show that
\begin{align*}
    \#\{\gamma_u:u \in \mathcal{E}_i'\} \gg \euler^{\star \ell_i} R^{-\star \kappa} \gg \euler^{\star \ell_i},
\end{align*}
when $\theta$ is sufficiently big in comparison to $\kappa$, see~\eqref{eq:define elli}

By Lemma~\ref{lem:bound on index} we have $[\Mbf_i:\Mbf_i^\circ] \ll_N 1$ since by construction the subgroup of rational points is dense in $\Mbf_i$.
Thus, there exists a finite subset $\mathcal{E}' \subset \mathcal{E}_i'$ with $\#\mathcal{E}' \gg \euler^{\star \ell_i}$ so that the lattice elements $\gamma_u$, $u \in \mathcal{E}'$, all belong to the same $\Mbf_i^\circ$ coset and are distinct. 
Thus, we have in view of \eqref{eq:heightboundlatticel} for $u_1 \in \mathcal{E}'$ fixed
\begin{align*}
\# \{\gamma \in \Mbf_i^\circ(\Z): \norm{\gamma} \leq \euler^{2\ref{a:boundlatticeel}\ell_i}\}
\geq \#\{\gamma_{u_1}^{-1}\gamma_u: u \in \mathcal{E}'\} \gg \euler^{\star \ell_i}
\end{align*}
proving the claimed polynomial amount of lattice elements.
\end{proof}

We have $\Mbf_0^\circ \subset \Mbf_1^\circ \subset \ldots \subset \Mbf_{\dim(\G)}^\circ$ and so there exists $i_0\in \{0,\ldots,\dim(\G)-1\}$ such that $\Mbf_{i_0}^\circ = \Mbf^{\circ}_{i_0+1}$.
We set for a choice of $i_0$
\begin{align}\label{eq:constrHgroup}
\Mbf = \Mbf_{i_0}^\mathcal{H}.
\end{align}
Using Proposition~\ref{prop:heightifgenbylattice} and \eqref{eq:heightboundlatticel} we have 
\begin{align}\label{eq:heightofHinproof}
\height(\Mbf) \ll \euler^{\star \ell_{i_0}}.
\end{align}
By Lemmas~\ref{lem:gen group proper} and \ref{lem:classHpart nontrivial}, $\Mbf$ is a non-trivial proper subgroup of $\G$.

\subsubsection{\textbf{\textit{Completion of the proof of Proposition~\ref{prop:mainstep}}}} 
Assume first that $\Mbf$ defined in \eqref{eq:constrHgroup} is not normal.
Note that for any $u \in \mathcal{E}_{i_0+1}'$ we have $\gamma_u.\vpz_M = \pm \vpz_M$. 
Indeed, since $\Mbf$ is normalized by $\Mbf_{i_0+1}$ by construction, $\gamma_u.\vpz_M$ is an integer multiple of $\vpz_M$ and that multiple is a unit since $\gamma_u.\vpz_M$ is also primitive.
So for any $u\in \mathcal{E}_{i_0+1}'$ using \eqref{eq:g0}, \eqref{eq:deflatticeelements}, \eqref{eq:heightofHinproof}, and $\kappa < \theta^{\dim(\G)+1}$
\begin{align*}
\norm{\eta_M(g_0 \lambda_k(w)^{-1}\lambda_k(u))}
= \norm{\eta_M(\gamma_u g_0 \lambda_k(w)^{-1}\lambda_k(u))} \ll R^{\star \kappa} \euler^{\star \ell_{i_0}} \ll \euler^{\star \ell_{i_0}}.
\end{align*}
Let 
\begin{align*}
    \mathcal{F}_1=\{u' \in B_U(e):\lambda_{\ell_{i_0+1}}(u')=\lambda_k(w)^{-1}\lambda_k(u) \text{ for some } u\in \mathcal{E}_{i_0+1}'\}.
\end{align*}
Then
\begin{align*}
\norm{\eta_M(g_0 \lambda_{\ell_{i_0+1}}(u'))} \ll \euler^{\star \ell_{i_0}}
\end{align*}
for all $u'\in \mathcal{F}_1$. 
We observe that by the definition \eqref{eq:Eprime_i} of $\mathcal{E}_{i_0+1}$, 
\begin{align*}
\mathcal{F}_1=\{u' \in B_U(e):\lambda_k(w)\lambda_{\ell_{i_0+1}}(u')=\lambda_k(u) \text{ for some } u\in \mathcal{E}\}.
\end{align*}
By \eqref{eq:gooddensitypts2}, we have $|\mathcal{F}_1|\gg |\mathcal{E}|$.
The Remez inequality in Lemma~\ref{lem:Remez} thus implies
\begin{align}\label{eq:fromEtolargeint}
\norm{\eta_M(g_0 \lambda_{\ell_{i_0+1}}(u))} \ll \euler^{\star \ell_{i_0}}
\end{align}
for all $u \in B_U(e)$. By \cite[Prop.~5.8]{LMMS}, we have for $\Lbf = \Nbf_{\Mbf}^{\mathcal{H}}$ (where $\Nbf_\Mbf$ is the normalizer of $\Mbf$ and it is a proper subgroup of $\Gbf$) and all $u \in B_U(e)$
\begin{align} \label{eq:Lg_0l_i0}
\begin{split}
\norm{\eta_L(g_0 \lambda_{\ell_{i_0+1}}(u))} &\ll |g_0|^\star \euler^{\star \ell_{i_0}}\ll \euler^{\star \ell_{i_0}},\\
\max_{\zpz \in \mathcal{B}_U}\norm{\zpz \wedge \eta_L(g_0 \lambda_{\ell_{i_0+1}}(u))} &\ll |g_0|^\star\euler^{\star \ell_{i_0}}\euler^{-\star \ell_{i_0+1}}\ll \euler^{-\star \ell_{i_0+1}},
\end{split}
\end{align}
because $|g_0| \ll R^{\star \kappa}$ by \eqref{eq:g0}, $\euler^{\ell_{i_0}}\geq R$ and $\ell_{i_0+1}\geq \lfloor\theta^{-1}\rfloor\ell_{i_0}$ by \eqref{eq:define elli}, and we choose sufficiently small $\kappa$ and $\theta$. 
Recall that $g_0 = \gamma_0 g \lambda_k(w)$ (cf.\ \eqref{eq:g0}). Therefore,  $\eta_L(g_0) = \eta_{\gamma_0^{-1}L\gamma_0}(g \lambda_{k}(w))$.  
Hence, by putting $u=e$ in \eqref{eq:Lg_0l_i0}, there exists an absolute constant $A>0$ such that
\begin{align*}
\begin{split}
\norm{\eta_{\gamma_0^{-1}L\gamma_0}(g \lambda_{k}(w))} &\ll \euler^{A \ell_{i_0}},\\
\max_{\zpz \in \mathcal{B}_U}\norm{\zpz \wedge \eta_{\gamma_0^{-1}L\gamma_0}(g \lambda_{k}(w))} &\ll \euler^{-\ell_{i_0+1}/A}.
\end{split}
\end{align*}
Taking $\ref{a:mainstep} \geq 2A$ and $\theta$ sufficiently small with $D \leq \frac{1}{4A^2\theta}$, this proves the conclusion \eqref{eq:mainstep} of the proposition for the point $w \in \mathcal{E}_0$ and the proper $\Q$-subgroup $\gamma_0^{-1} \Lbf \gamma_0<\Gbf$ of class $\mathcal{H}$, when the subgroup $\Mbf$ associated to $w$ is not normal.

\medskip

Suppose now that the subgroup $\Mbf$ associated to $w$ \textbf{is} a normal subgroup of $\G$.
In this case, the above argument fails since, of course, the necessary passage to the normalizer does not produce a proper subgroup. 
Hence, to extract information from~\eqref{eq:deflatticeelements}, we will instead use the Chevalley representation for $\bf M$ constructed in \S\ref{sec:chevalleyclassH}. 

Let $(\rho,\vpz)$ be a Chevalley pair for $\Mbf$ as in Proposition~\ref{prop:chevalley}.
So, by \eqref{eq:heightofHinproof},
\begin{align*}
    \norm{\vpz} \ll \height(\Mbf)^\star \ll \euler^{\star \ell_{i_0}}.
\end{align*}

Let $V\subset \Q^{\dim(\rho)}$ be the subspace of $\Mbf(\Q)$-fixed vectors.
The identity component of the image of $\Mbf_{i_0+1}$ under the restricted representation $\rho'\colon \G \to \GL(V)$ is a {$\Q$-torus}.
Let $\mathcal{S}=\{\rho'(\gamma_u): u \in \mathcal{E}_{i_0+1}'\}$. Then
by Lemma~\ref{lem:bound on index},  \eqref{eq:heightboundlatticel}, and \eqref{eq:latticeptcountintori}  we have  
\begin{align*}
    \#(\mathcal{S})\ll \ell_{i_0}^\star \ll \log(R)^\star.
\end{align*}
For each $s\in\mathcal{S}$, let 
\begin{align}
    \mathcal{E}'(s)&=\{u\in \mathcal{E}'_{\ell_{i_0+1}}:\rho'(\gamma_u)=s\} \text{ and } \notag\\
    \mathcal{F}(s)&=\{u\in B_U(e):\lambda_k(w)\lambda_{\ell_{i_0+1}}(u)\in \lambda_k(\mathcal{E}'(s))\}. 
    \label{eq:Fs}
\end{align}
Then  $\cup_{s\in\mathcal{S}}\mathcal{E}'(s)=\mathcal{E}_{\ell_{i_0+1}}'$, and hence 
\begin{align*}
    \cup_{s\in \mathcal{S}}\mathcal{F}(s)&=\{u\in B_U(e):\lambda_k(w)\lambda_{\ell_{i_0+1}}(u)\in \lambda_k(\mathcal{E}_{\ell_{i_0+1}}')\}\\
    &=\{u\in B_U(e):\lambda_k(w)\lambda_{\ell_{i_0+1}}(u)\in \lambda_k(\mathcal{E})\},
\end{align*}
by the definition of $\mathcal{E}_{\ell_{i_0+1}}'$. So, by \eqref{eq:gooddensitypts2} we have 
\begin{align*}
    |\cup_{s\in \mathcal{S}}\mathcal{F}(s)|\gg |\mathcal{E}|. 
\end{align*}
By the pigeonhole principle, we can pick $s_1\in \mathcal{S}$ such that
\begin{align*}
  |\mathcal{F}(s_1)|\gg |\mathcal{E}|/\#(\mathcal{S})\gg |\mathcal{E}|\log(R)^{-\star}. 
\end{align*}

We fix $u_1 \in \mathcal{E}'(s_1)$. For simplicity, we denote the right orbit map at $\vpz$ by $\vartheta: g \in \G \mapsto \rho(g)^{-1}\vpz$. For any $u \in \mathcal{E}'(s_1)$, since $\vartheta(\gamma_{u})=\vartheta(\gamma_{u_1})$, we have
\begin{align} \label{eq:gugu0}
\vartheta(g_0 g_u) = \vartheta(\gamma_{u_1}g_0 \lambda_k(w)^{-1}\lambda_k(u))
= \vartheta(g_0 g_{u_1} \lambda_k(u_1)^{-1}\lambda_k(u))
\end{align}
where $g_u = g_0^{-1} \gamma_u g_0 \lambda_k(w)^{-1}\lambda_k(u)$ as in \eqref{eq:gu}. Also, by \eqref{eq:g0} and \eqref{eq:deflatticeelements}, since $\norm{\vpz}\ll e^{\star \ell_{i_0}}$, we have
\begin{align} \label{eq:bondong0gu}
    \norm{\vartheta(g_0 g_u)}\ll \euler^{\star \ell_{i_0}}.
\end{align} 
Let $u_1'\in B_U(e)$ be such that $\lambda_k(u_1) = \lambda_k(w)\lambda_{\ell_{i_0+1}}(u_1')$. Then,
\begin{align}
    \lambda_k(\mathcal{E}'(s_1))&\subset\lambda_k(\mathcal{E}_{\ell_{i_0}+1}') \notag\\
    &\subset \lambda_k(w)\lambda_{\ell_{i_0}+1}(B_U(e))  \text{, by definition of $\mathcal{E}_{\ell_{i_0+1}}'$,}\notag\\
    &=\lambda_k(u_1)\lambda_{\ell_{i_0+1}}(u_1')^{-1}\lambda_{\ell_{i_0}+1}(B_U(e))\notag\\
    &\subset \lambda_k(u_1)\lambda_{\ell_{i_0+1}}(\lambda_{k_0}(B_U(e))), \label{eq:k0Bue}
\end{align}
by \eqref{eq:lambda-scale}. So, we define
\begin{align} 
    \mathcal{F}_2=\{u'\in \lambda_{k_0}(B_U(e)):\lambda_k(u_1)\lambda_{\ell_{i_0+1}}(u')\in \lambda_k(\mathcal{E}'(s_1))\}.
\end{align}
Then, by \eqref{eq:gugu0} and \eqref{eq:bondong0gu}, for all $u'\in \mathcal{F}_2$, we have
\begin{align} \label{eq:F2bound}
    \vartheta(g_0 g_{u_1} \lambda_{\ell_{i_0}+1}(u'))\ll \euler^{\star \ell_{i_0}}.
\end{align}
Also,
\begin{align*}
    \mathcal{F}_2&=\{u'\in \lambda_{k_0}(B_U(e)): \lambda_k(w)\lambda_{\ell_{i_0}+1}(u_1')\lambda_{\ell_{i_0+1}}(u')\in \lambda_k(\mathcal{E}'(s_1))\}\\
    &\supset\{u'\in \lambda_{k_0}(B_U(e)):\lambda_{\ell_{i_0}+1}(u_1')\lambda_{\ell_{i_0}+1}(u')\in\lambda_{\ell_{i_0+1}}(\mathcal{F}(s))\} \text{, by \eqref{eq:Fs},}\\ 
    &=\lambda_{\ell_{i_0+1}}^{-1}\left(\lambda_{\ell_{i_0}+1}(u_1')^{-1}\lambda_{\ell_{i_0+1}}(\mathcal{F}(s))\right).
    \end{align*}
Since $\lambda_{\ell_{i_0+1}}$ acts on the Haar measure of $U$ as a scalar, 
\begin{align} \label{eq:F2vol}
     |\mathcal{F}_2|\geq |\mathcal{F}(s)| \gg|\mathcal{E}|(\log R)^{-\star}.
\end{align}
Therefore, from \eqref{eq:F2bound} using the Remez inequality (Lemma~\ref{lem:Remez}), we conclude that 
\begin{align}\label{eq:orbit bounded}
    \norm{\vartheta(g_0 g_{u_1}\lambda_{\ell_{i_0+1}}(u))} \ll \euler^{\star \ell_{i_0}}
\end{align}
for all $u \in B_U(e)\subset\lambda_{k_0}(B_U(e))$.

Now we will argue as in \cite[Prop.~5.8]{LMMS} to conclude. Fix $\zpz \in \mathcal{B}_U$.
The estimate in \eqref{eq:orbit bounded} implies that 
\begin{align*}
\norm{\vartheta(g_0 g_{u_1}\lambda_{\ell_{i_0+1}}(\exp(tT\hat{\zpz})))} \ll \euler^{\star \ell_{i_0}} \text{ for all $|t| \leq 1$,}
\end{align*}
where $\hat{\zpz} = \lambda^{-\ell_{i_0+1}}(\zpz)$ and $T = \norm{\hat{\zpz}}^{-1}$ (recall that by definition $\lambda_\ell(\exp(\zpz))=\exp(\lambda^\ell(\zpz))$ for all $\ell$).
Thus, using \eqref{eq:deflatticeelements} and $|g_0|\ll R^{\star \kappa}$, we get
\begin{align} \label{eq:poly-bound}
\norm{\vartheta(\exp(t \Ad(g_0g_{u_1})\zpz))} \ll \euler^{\star \ell_{i_0}}\text{ for all $|t| \leq T$}.
\end{align}
By Proposition~\ref{prop:chevalley}, each coordinate function of the map 
\begin{align*}
    t\mapsto \vartheta(\exp(t \Ad(g_0g_{u_1})\zpz))
\end{align*}
is a polynomial of degree bounded by a number that depends only on $N$. Hence, from \eqref{eq:poly-bound}, using Lagrange interpolation, we conclude that all of the non-constant coefficients of the polynomial map must be $\ll \euler^{\star \ell_{i_0}} T^{-1}$ in size. 

Using $T \gg (\norm{\lambda^{-1}}^{-1})^{\ell_{i_0+1}}$, we obtain that
\begin{align*}
\norm{\mathrm{D}\rho(\Ad(g_0g_{u_1})\zpz)\vpz} \ll \euler^{\star \ell_{i_0}} T^{-\star} \ll \euler^{-\star \ell_{i_0+1}},
\end{align*}
because $\ell_{i_0+1}\geq \lfloor \theta^{-1}\rfloor \ell_{i_0}$, and we choose the constant $\theta$ sufficiently small. 

By definition of the Chevalley representation, 
\begin{align*}
\Lie(\Mbf) = \{\wpz \in \mathfrak{sl}_N: \mathrm{D}\rho(\wpz)\vpz = 0\}.
\end{align*}
The map $\wpz \mapsto \mathrm{D}\rho(\wpz)\vpz$ is linear and can be realized as an integral matrix with coefficients of size $\ll \height(\Mbf)^\star \ll \euler^{\star \ell_{i_0}}$.
In view of the above estimate, there exists $\zpz'$ in the kernel $\Lie(\Mbf)$ with distance $\ll \euler^{-\star \ell_{i_0+1}}$ to $\Ad(g_0g_{u_1})\zpz$ (see e.g.~\cite[\S13.4]{EMV}) and hence
\begin{align*}
\norm{\Ad(g_0g_{u_1})\zpz \wedge \vpz_{M}} \ll \euler^{-\star \ell_{i_0+1}}.
\end{align*}
As $\Mbf$ is normal, this implies $\norm{\zpz \wedge \vpz_{M}} \ll \euler^{-\star \ell_{i_0+1}}$, because $|g_0g_{u_1}| \ll \euler^{\star \ell_{i_0}}$, $\ell_{i_0+1}\geq \lfloor \theta^{-1}\rfloor \ell_{i_0}$, and we choose $\theta$ sufficiently small. 
For appropriate $\ref{a:mainstep}$ and $\theta$, as in the non-normal case considered earlier, and noting that $\eta_M(\cdot) = \vpz_M$, this completes the proof of Proposition~\ref{prop:mainstep} in combination with \eqref{eq:heightofHinproof}.
\qed

\subsection{Proof of Theorem \ref{thm:main}}

We begin by recalling the notion of Diophantine points as well as the main result from \cite{LMMS}.

\begin{definition}[\cite{LMMS}]
Let $\epsilon: \R_{> 0} \to (0,1)$ be a monotonely decreasing function.
A point $x = \Gamma g$ is called $(\epsilon,t)$-Diophantine (with respect to $U$) if for any non-trivial proper subgroup $\Mbf < \G$ of class $\mathcal{H}$ with $\norm{\eta_M(g)} \leq \mathrm{e}^t$ we have
\begin{align*}
\max_{\zpz \in \mathcal{B}_U}\norm{\zpz \wedge \eta_M(g)} \geq \epsilon(\norm{\eta_M(g)}).
\end{align*}
\end{definition}

\noindent
Using this notion, we can state the main result of \cite{LMMS}: 

\begin{theorem}[{\cite[Thm.~3.3]{LMMS}}]\label{thm:efflin}
There exist constants $\consta\label{a:efflin},\consta\label{a:efflin2}>0$ depending only on $N$, $E>0$ depending on N and polynomially on $\height(\G)$, and $E_1$ depending in addition also (polynomially) on $E_{\G}$, so that the following holds.
Let $g \in G$, $t>0$, $k \geq 1$, $\eta \in (0,1/2)$. 
Assume $\epsilon: \R_{>0} \to (0,1)$ satisfies for any $s \in \R_{>0}$ that
\begin{align}\label{eq:assump eps}
\epsilon(s) \leq s^{-\ref{a:efflin}} \eta^{\ref{a:efflin}}/E_1.
\end{align}
Then, at least one of the following holds.
\begin{enumerate}
    \item
    \begin{align*}
    \Big|\Big\{u \in B_U(e):
        \begin{array}{cc}
             \Gamma g \lambda_k(u) \not\in X_\eta \text{ or} \\
             \Gamma g \lambda_k(u) \text{ is not }(\epsilon,t)\text{-Diophantine}
        \end{array}\Big\}\Big| < E_1 \eta^{1/\ref{a:efflin2}}.
    \end{align*}
    \item There exists a nontrivial proper subgroup $\Mbf \in \mathcal{H}$ so that for all $u \in B_U(e)$
    \begin{align*}
    \norm{\eta_M(g\lambda_k(u))} &\leq (E |g|^{\ref{a:efflin}} + E_1 \euler^{\ref{a:efflin}t}) \eta^{-\ref{a:efflin}},\\
    \max_{\zpz \in \mathcal{B}_U} \norm{\zpz \wedge \eta_M(g\lambda_k(u))} 
    &\leq \euler^{-k/\ref{a:efflin2}}(E |g|^{\ref{a:efflin}} + E_1 \euler^{\ref{a:efflin}t}) \eta^{-\ref{a:efflin}}.
    \end{align*}
    \item There exists a nontrivial proper normal subgroup $\Mbf \lhd \G$ with
    \begin{align*}
    \height(\Mbf) &\leq E_1 (\euler^t \eta^{-1})^{\ref{a:efflin}},\\
    \max_{\zpz \in \mathcal{B}_U} \norm{\zpz \wedge \vpz_{M}} 
    &\leq \epsilon(\height(\Mbf)^{1/\ref{a:efflin}}\eta/E_1)^{1/\ref{a:efflin}}.
    \end{align*}
\end{enumerate}
\end{theorem}

\begin{proof}[Proof of Theorem~\ref{thm:main}]
Fix a sufficiently large constant $D>1$ to be determined later, and let $\mathcal{E}_0$ be as in Proposition~\ref{prop:mainstep}.
We may reduce $\mathcal{E}_0$ further and assume that \eqref{eq:mainstep} holds for all $u \in \mathcal{E}_0$ and some \emph{fixed} $i \in \{0,\ldots,\dim(\G)\}$.
Set for $s >0$
\begin{align*}
\epsilon(s) = \euler^{-D\ell_{i}\ref{a:mainstep}/2} s^{-\ref{a:efflin}}.
\end{align*}
Note that $\epsilon(\cdot)$ satisfies \eqref{eq:assump eps} whenever $\eta \geq \euler^{-\ell_{i}}$ and $D$ is sufficiently large.

Observe now that for any $u \in \mathcal{E}_0$ the point $x \lambda_k(u)$ is not $(\epsilon,\ref{a:mainstep}\ell_i)$-Diophantine.
Indeed, by Proposition~\ref{prop:mainstep} there exists for any $u \in \mathcal{E}_0$ a nontrivial proper $\Q$-subgroup $\Lbf \in \mathcal{H}$ of $\G$ such that
\[
\norm{\eta_L(g\lambda_k(u))} \leq \euler^{\ref{a:mainstep} \ell_i},
\]
see the first estimate in \eqref{eq:mainstep}, and moreover, by the second estimate in \eqref{eq:mainstep} whenever $D>2\ref{a:efflin}$
\begin{align*}
\max_{\zpz \in \mathcal{B}_U} \norm{\zpz \wedge \eta_L(g\lambda_k(u))} 
&\leq \euler^{-D\ell_{i}\ref{a:mainstep}}\\
&< \euler^{-D\ell_{i}\ref{a:mainstep}/2 } \euler^{-\ref{a:efflin}\ref{a:mainstep}\ell_i}
\leq \epsilon(\norm{\eta_L(g\lambda_k(u))}).
\end{align*}
We may apply Theorem~\ref{thm:efflin} for $\eta >0$ with $E_1 \eta^{1/\ref{a:efflin2}} = R^{-2\kappa}$ assuming $\kappa$ is small enough so that $\theta^{\dim(\G)+1} > 2\ref{a:efflin2}\kappa$.
Since $|\mathcal{E}_0| \geq R^{-2\kappa}$, we obtain that Option (2) or Option (3) in Theorem~\ref{thm:efflin} holds.
If (2) holds, we conclude with Option \ref{item:main-option1} in Theorem~\ref{thm:main}.
So assume that Option (3) of Theorem~\ref{thm:efflin} holds for some nontrivial proper normal subgroup $\Mbf \lhd \G$. In particular,
\begin{align}\label{eq:nildirclosetonormal}
\max_{\zpz \in \mathcal{B}_U} \norm{\zpz \wedge \vpz_{M}} 
\leq \epsilon(\height(\Mbf)^{1/\ref{a:efflin}}\eta/E_1)^{1/\ref{a:efflin}} \ll \euler^{-D \ell_{i}\ref{a:mainstep}/3}.
\end{align}
by definition of $\epsilon(\cdot)$.
By Lemma~\ref{lem:idealscontainingLevi}, $\Lie(\Mbf)$ cannot contain a Levi subalgebra of $\gfrak$ and, in particular, $\Mbf' = \Mbf \cdot\Rbf(\G)$ is a proper normal subgroup of $\G$ where $\Rbf(\G)$ denotes the (unipotent) radical of $\G$.
Clearly, the distance $\zpz$ to $\Lie(M')$ is at most the distance of $\zpz$ to $\Lie(M)$.
These statements together show that $\norm{\zpz \wedge \vpz_{M'}}\ll \norm{\zpz \wedge \vpz_{M}}$ for all $\zpz \in \mathcal{B}_U$ and thus Option~(2) of Theorem~\ref{thm:main} follows from \eqref{eq:nildirclosetonormal}. 

Moreover, we also note that there are only finitely many choices for $\Mbf'$, and hence $\height(\Mbf')\ll_\G 1$ (see \cite[Lemma 8.6]{AW-realsemisimple} for a more precise estimate).
\end{proof}

\bibliographystyle{plain}
\bibliography{papers}

\end{document}

%% file: environments.tex
\newtheorem{theorem}{Theorem}
\newtheorem{lemma}[theorem]{Lemma}
\newtheorem{definition}[theorem]{Definition}
\newtheorem{proposition}[theorem]{Proposition}

\theoremstyle{remark}

\theoremstyle{definition}

\numberwithin{theorem}{section}
\numberwithin{equation}{section}

%% file: macros.tex
\newcommand{\C}{\mathbb{C}}
\newcommand{\R}{\mathbb{R}}
\newcommand{\Q}{\mathbb{Q}}
\newcommand{\Z}{\mathbb{Z}}
\newcommand{\N}{\mathbb{N}}

\newcommand{\Mat}{\mathrm{Mat}}

\newcommand{\Lie}{\mathrm{Lie}}

\newcommand{\G}{\mathbf{G}} 
\newcommand{\GL}{\mathrm{GL}}

\newcommand{\SL}{\mathrm{SL}}

\newcommand{\wholenorm}[1]{\mathbf{N}_{#1}}

\newcommand{\id}{\mathrm{id}}

\newcommand{\height}{\mathrm{ht}} 
\newcommand{\Ad}{\mathrm{Ad}}
\newcommand{\ad}{\mathrm{ad}}

\newcommand{\metric}{\mathrm{d}}
\newcommand{\norm}[1]{\|#1\|}
\newcommand{\abs}[1]{|{#1}|}

\newcommand{\Hcal}{\mathcal{H}}

\newcommand{\gfrak}{\mathfrak{g}}
\newcommand{\hfrak}{\mathfrak{h}}

\newcommand{\mfrak}{\mathfrak{m}}

\newcommand{\sfrak}{\mathfrak{s}}

\newcommand{\ufrak}{\mathfrak{u}}

\newcommand{\Abf}{\mathbf{A}}

\newcommand{\Cbf}{\mathbf{C}}

\newcommand{\Gbf}{\mathbf{G}}

\newcommand{\Lbf}{\mathbf{L}}
\newcommand{\Mbf}{\mathbf{M}}
\newcommand{\Nbf}{\mathbf{N}}

\newcommand{\Rbf}{\mathbf{R}}

\newcommand{\Tbf}{\mathbf{T}}
\newcommand{\Ubf}{\mathbf{U}}

\def\vpz{\mathpzc{v}}
\def\upz{\mathpzc{u}}

\def\wpz{\mathpzc{w}}

\def\zpz{\mathpzc{z}}

\newcommand\rquot[2]{
  \mathchoice
  {
    \text{\raise0.5ex\hbox{$#1$}\big/\lower0.5ex\hbox{$#2$}}%
  }
  {
    #1\,/\,#2
  }
  {
    #1\,/\,#2
  }
  {
    #1\,/\,#2
  }
}

\newcommand\lquot[2]{
  \mathchoice
  {
    \text{\lower0.5ex\hbox{$#1$}\big\backslash\raise0.5ex\hbox{$#2$}}%
  }
  {
    #1\,\backslash\,#2
  }
  {
    #1\,\backslash\,#2
  }
  {
    #1\,\backslash\,#2
  }
}

\newcommand\lrquot[3]{
  \mathchoice
  {
    \text{\lower0.5ex\hbox{$#1$}\big\backslash\raise0.5ex\hbox{$#2$}\big/
      \lower0.5ex\hbox{$#3$}}%
  }
  {
    #1\,\backslash\,#2\,/\,#3
  }
  {
    #1\,\backslash\,#2\,/\,#3
  }
  {
    #1\,\backslash\,#2\,/\,#3
  }
}

\newcounter{consta}
\renewcommand{\theconsta}{{A_{\arabic{consta}}}}
\newcommand{\consta}{\refstepcounter{consta}{\color{red}\theconsta}}

\newcounter{constk}

\newcounter{constc}
\renewcommand{\theconstc}{{c_{\arabic{constc}}}}
\newcommand{\constc}{\refstepcounter{constc}{\color{red}\theconstc}}

\newcounter{constE}
\renewcommand{\theconstE}{{{C}_{\arabic{constE}}}}
\newcommand{\constE}{\refstepcounter{constE}{\color{red}\theconstE}}

\newcounter{constd}

\newcommand{\be}{\begin{equation}}
\newcommand{\ee}{\end{equation}}

\newcommand{\euler}{\mathrm{e}}